\documentclass[10pt,fullpage]{article}

\usepackage{amssymb}
\usepackage{amsmath}
\usepackage{amsthm}
\usepackage{epsfig}
\usepackage{enumerate}
\usepackage{multicol}
\usepackage{caption}
\usepackage{psfrag}
\usepackage{xcolor}
\usepackage{graphicx}
\usepackage{authblk}
\usepackage{fullpage}

\input epsf

\numberwithin{equation}{section}

\newtheorem{Lemma}{Lemma}[section]

\newtheorem{Th}[Lemma]{Theorem}
\newtheorem{Prop}[Lemma]{Proposition}
\newtheorem{Def}[Lemma]{Definition}
\newtheorem{claim}[Lemma]{Claim}
\theoremstyle{remark}
\newtheorem{Rk}[Lemma]{Remark}

\newcommand{\be}{\begin{equation}}
\newcommand{\ee}{\end{equation}}
\newcommand{\baa}{\begin{array}}
\newcommand{\eaa}{\end{array}}
\newcommand{\ba}{\begin{eqnarray}}
\newcommand{\ea}{\end{eqnarray}}

\def\epsilon{\varepsilon}

\def\R{{\mathbb R}}

\newcommand{\bee}{\begin{equation*}}
\newcommand{\eee}{\end{equation*}}
\newcommand{\bc}{\begin{cases}}
\newcommand{\ec}{\end{cases}}

\date{}
\title{Extinction and spreading of a species under the joint influence of climate change and a weak Allee effect: a two-patch model}
\author[1]{Juliette Bouhours}
\author[2]{Thomas Giletti}
\affil[1]{Departments of Mathematical and Statistical Sciences and of Biological Sciences,
University of Alberta, Edmonton, AB T6G 2G1, Canada}
\affil[2]{Institut Elie Cartan de Lorraine,
Universit\'{e} de Lorraine, Vandoeuvre-l\`{e}s-Nancy, France}
\begin{document}
\maketitle
\begin{abstract}

Many species see their range shifted poleward in response to global warming and need to keep pace in order to survive. To understand the effect of climate change on species ranges and its
consequences on population dynamics, we consider a space-time heterogeneous reaction-diffusion equation in dimension 1, whose unknown~$u (t,x)$ stands for a population density. 
More precisely, the environment consists of two patches moving with a constant climate shift speed $c \geq 0$: in the invading patch $\{ t >0 , \, x \in \R \, | \ x < ct \}$ the growth rate is negative and, in the receding patch $\{ t >0 , \, x \in \R \, | \ x \geq ct \}$ it is of the classical monostable type. 
Our framework includes species subject to a weak Allee effect, meaning that there may be a positive correlation between population size and its per capita growth rate. We study the large-time
behaviour of solutions in the moving frame and show that whether the population spreads or goes 
extinct depends not only on the speed $c$ but also, in some intermediate speed range, on the 
initial datum. This is in sharp contrast with the so-called `hair-trigger effect' in the 
homogeneous monostable equation, and suggests that the size of the population becomes a decisive factor under the joint influence of climate change and a weak Allee effect. Furthermore, our analysis exhibit sharp thresholds between spreading and extinction: in particular, we prove the 
existence of a threshold shifting speed which depends on the initial population, such that 
spreading occurs at lower speeds and extinction occurs at faster speeds.
\end{abstract}
\section{Introduction} 

In this paper we are interested in the following problem
\be\label{problemmono}\begin{cases}
\partial_t u-\partial_{xx}u=f(x-ct,u), &t\in(0,+\infty), \: x\in\R,\\
u(0,x)=u_0(x),& x\in\R,
\end{cases}\ee
where $c \geq 0$ and $u_0$ will be assumed to be a bounded, nonnegative and compactly supported function, and  
\be\label{f}
f(z,s)=\begin{cases}-s, &\text{if } z<0,\\
				g(s), &\text{if }z\geq 0.
				\end{cases}
\ee
Here $g \in  C^{1,r}_{loc} (\R , \R)$ is monostable in the sense that 
\be\label{monostable}
g(0)=g(1)=0, \quad g(s)>0 \quad \forall s\in(0,1), \quad g(s) < 0 \quad \forall s \not \in [0,1] , \quad g'(1)<0 < g'(0).\ee

This problem is motivated by the study of the effect of climate change on population persistence. In~\cite{MM00,WPetal02}, the different authors point out that global warming induces a shift of the climate envelope of some species toward higher altitude or latitude. Therefore, these species need to keep pace with their moving favourable habitat in order to survive. Here we model the evolution of the density~$u$ of a population using a reaction-diffusion equation where the reaction term, which accounts for the growth of the species, is heterogeneous with respect to both space and time. That is, we assume that the dispersion of the population is a diffusion process, and in~\eqref{problemmono} the growth term~$f$ depends on the shifting variable $x - ct$ where the nonnegative parameter~$c$ (which we assume to be constant) stands for the speed of the shift of the climate envelope. More precisely, while the growth term is monostable in the favourable environment \{$x- ct \geq 0\}$, it is negative in the unfavourable zone $\{x-ct < 0 \}$. We highlight here that by monostable, we only mean as stated in \eqref{monostable} that the growth rate of the population is positive when $u \in (0,1)$ (after renormalization, 1 stands for the carrying capacity in the climate envelope). However, we do not assume that the per capita growth rate is maximal at $u=0$, that is the population may be subjected to a weak Allee effect, which is a common feature of ecological species arising for instance from cooperative behaviours such as defense against predation~\cite{allee}. Let us also briefly note that our analysis easily carry out if the growth term $f$ in \eqref{f} is chosen to be any linear function $-\rho u$ with $\rho >0$ in the unfavourable patch $\{x -ct <0\}$: indeed our main results would still hold as they stand below, and we choose $\rho = 1$ here just for the sake of simplicity.

Such shifting range models were introduced in several papers to study similar ecological problems. In \cite{PL} and \cite{BDD}, the effect of climate change and shifting habitat on invasiveness properties is studied for a system of reaction-diffusion equations. In~\cite{BDNZ,BR1,BR2,Lietal14,Vo}, the authors deal with persistence properties under a shifting climate for a scalar reaction-diffusion equation in dimension 1 and higher, and exhibit a critical threshold for the shifting speed below which the species survives and above which it goes extinct. This problem has also been studied in the framework of integrodifference equations, where time is assumed to be discrete and dispersion is nonlocal~\cite{HZetal14,ZK,ZK13}. However, all the aforementioned papers hold under KPP type assumptions whose ecological meaning is that there is no Allee effect and which mathematically imply that the behaviour of solutions is dictated by a linearized problem around the invaded unstable state. In the context of~\eqref{problemmono}, the KPP assumption typically writes as $f(z,s) \leq \partial_s f(z,0) s$ for all $z \in \R$ and $s \geq 0$.

On the other hand a few papers investigate the case when no KPP assumption is made. In \cite{BRRK}, the authors analyse numerically the effect of climate change and the geometry of the habitat in dimension 2, again in the framework of reaction-diffusion equations, with or without Allee effect. In \cite{BN}, Bouhours and Nadin consider~\eqref{problemmono} when the size of the favourable zone is bounded under rather general assumptions on the reaction~$f$ in the favourable zone (including the classical monostable and bistable cases). More precisely, they have shown the existence of two speeds $\underline{c} < \overline{c}$ such that the population persists for large enough initial data when $c < \underline{c}$ and goes extinct when $c > \overline{c}$. However, for a fixed initial datum, it is not known in general whether there exists a threshold speed delimiting persistence and extinction. We will prove here that, when $f$ satisfies the hypotheses~\eqref{f} and~\eqref{monostable}, the answer is positive but also that, unlike in the KPP case~\cite{BDNZ,BR1,BR2,Lietal14,Vo}, the threshold speed depends non trivially on the initial datum.\\

Before stating our main results, let us start by giving some basic properties of problem~\eqref{problemmono}. Using sub and super solution method and comparison principle we know that, for bounded and nonnegative initial conditions, problem \eqref{problemmono} has a unique global and bounded solution. Because of the discontinuity of $f$ at $x=ct$, by solution we will always mean a function $u (t,x)$ which is $C^1$ with respect to~$x \in \R$ for all $t >0$, and which satisfies the equation in a classical sense on both $\{ t >0 \mbox{ and } x  < ct\}$ and $\{ t >0 \mbox{ and } x > ct \}$. The initial condition is understood as follows:
$$\underset{t\to0}{\lim}\:u(t,x)= u_0(x), \text{ for almost every } x\in\R.$$ 
In this paper we are interested in the asymptotic behaviour of the solution $u(t,x)$ of \eqref{problemmono} as time goes to infinity, and in particular whether the solution goes extinct or spreads, in the sense of Definition~\ref{def:main} below. To do so we will study the previous problem in the moving frame, i.e. letting $z:=x-ct$. If we define $u^c(t,z):=u(t,z+ct)$ to be the solution in the moving frame, then $u^c$ is solution of the parabolic equation
\be\label{problemmonobis}
\partial_t u-\partial_{zz}u-c\partial_z u=f(z,u), \quad t\in(0,+\infty), \: z\in\R,\\
\ee
such that $u^c(0,z)=u_0(z)$, for all $z\in\R$.

\subsubsection*{Notation}

In the following we will denote by $z$ the space variable in the moving frame and $x$ the variable in the non moving frame. With some slight abuse of notations, we will write $u(t,x)$ when we consider the solution in the non moving frame and $u(t,z)$ (instead of $u^c(t,z)$) for the solution in the moving frame. 

Let us also introduce the classical homogeneous monostable equation
\be\label{problemhomo1}
\partial_t v-\partial_{xx}v =g(v), \quad t\in(0,+\infty), \: x\in\R ,\\
\ee
and denote by $c ^* $ the minimal speed for existence of travelling wave solutions of \eqref{problemhomo1}, namely particular solutions of the type $v (t,x) = V (x- ct)$ where $V ( - \infty) = 1 > V (\cdot) > V (+\infty)=0$. It is well-known that $c^* \geq 2 \sqrt{g'(0)}$ and that, under the KPP hypothesis $g(s) \leq g'(0) s$ for all $s \geq 0$, then $c^* = 2 \sqrt{g'(0)}$. We refer to the seminal paper of Aronson and Weinberger~\cite{AW} for details.

\subsection{Main results}\label{sec:main_res}

As we are interested in the long time behaviour of the solution of problem \eqref{problemmonobis}, a first step is to classify the stationary solutions of \eqref{problemmonobis}. 
As we will see more precisely in Theorem \ref{th:stationary} in Section \ref{sec:stationarysol}, there are three types of nonnegative stationary solutions of \eqref{problemmonobis}:
\begin{itemize}
\item the trivial solution 0;
\item `ground states', namely positive stationary solutions $p$ of \eqref{problemmonobis} such that $p (\pm \infty) = 0$, among which there is a `critical ground state' which is characterized by the fact that it has the largest value at $z=0$ as well as the fastest decay as $z \to +\infty$;
\item a unique `invasion state', namely a positive stationary solution $p_+^c$ of \eqref{problemmonobis} such that $p_+^c (-\infty)=0$ and $p_+^c (+\infty)=1$.
\end{itemize}
Let us highlight here that these are stationary solutions in the moving frame with the same speed~$c$ as the shifting favourable zone, which may thus be seen as travelling waves in the original non moving frame. In particular, while we refer to $p_+^c$ as the `invasion state', the invasion is of course restricted to the favourable zone. In a similar fashion, by `ground state' we mean that the population migrates but does not spread away from the point $x =ct$.

While the invasion state always exists, there may exist either none or an infinity of ground states depending on the value of $c$. In the latter case, ground states all lie below the invasion state, and each of them will be characterized by its value at $z=0$. We will denote by $p_\alpha$ the ground state such that $p_\alpha(0)=\alpha$. The supremum of the admissible $\alpha$ such that $p_\alpha$ exists will be denoted by $\alpha^*_c$ and, as mentioned above, we call $p_{\alpha^*_c}$ the critical ground state. Proposition~\ref{prop:exp_decay} describes the asymptotics of the ground states $p_\alpha (z)$ as $z \to +\infty$ depending on $\alpha$, which will justify our statement that the critical ground state has the fastest decay at infinity. The discussion above will be made more rigorous in Section~\ref{sec:stationarysol}.

From this, we can be more precise about what we mean by extinction or spreading of the solution $u(t,z)$ of \eqref{problemmonobis}:
\begin{Def}\label{def:main} Let $u_0$ be a nonnegative, bounded and compactly supported initial datum, and $u(t,z)$ be the solution of \eqref{problemmonobis} with initial condition $u_0$. We say that
\begin{itemize}
\item \textbf{extinction} occurs if $u (t,z)$ converges uniformly with respect to $z$ to 0 as time~$t$ goes to infinity;
\item \textbf{grounding} occurs if $u (t,z)$ converges uniformly with respect to $z$ to the critical ground state as time~$t$ goes to infinity;
\item \textbf{spreading} occurs if $u (t,z)$ converges locally uniformly to the invasion state as time goes to infinity.
\end{itemize}
\end{Def}
Note that the fact that the convergence is or is not uniform is a consequence of the choice of compactly supported initial data (clearly the solution may never converge uniformly to the invasion state). Furthermore, we narrow grounding to the large-time convergence to the critical ground state (we again refer to Theorem~\ref{th:stationary} for a more precise definition). The reason is that our arguments will largely rely on the fact, which is contained in our main statements below and which will be proved in Section~\ref{sec:convstatsol}, that the solution may never converge to a non critical ground state when the initial condition $u_0$ is compactly supported.\\

Our first theorem investigates the large-time behaviour of the solution of \eqref{problemmonobis} in various speed ranges:
\begin{Th}\label{th:regime}
Let $u$ be the solution of~\eqref{problemmonobis} with a nonnegative, non trivial, bounded and compactly supported initial datum~$u_0$.
\begin{enumerate}[$(i)$]
\item If $0 \leq c < 2 \sqrt{g'(0)}$, then spreading occurs.
\item If $2\sqrt{g'(0)} \leq c < c^*$ (provided such $c$ exists), then both spreading and extintion may occur depending on the choice of $u_0$. To be more precise, there exist initial data $u_{0,1} > u_{0,2}$ such that spreading occurs for $u_{0,1}$, and extinction occurs for $u_{0,2}$.
\item If $c \geq c^*$, then extinction occurs.
\end{enumerate}
\end{Th}
Recall from~\cite{AW} that the minimal wave speed of~\eqref{problemhomo1} is also the spreading speed of solutions of the Cauchy problem with compactly supported initial data. Therefore, the third statement of Theorem~\ref{th:regime} simply means that, when the climate shifts faster than the species spreads in a favourable environment, then the species cannot keep pace with its climate envelope and goes extinct as time goes to infinity. On the other hand, when $c$ is less than $2 \sqrt{g'(0)}$ (statement $(i)$ of Theorem~\ref{th:regime}) which is the speed associated with the linearized problem around $u=0$, then any small population is able to follow its habitat and thrive. In particular, when $c^* = 2 \sqrt{g'(0)}$, we retrieve a threshold speed between persistence and extinction, which as in the KPP framework of~\cite{BDNZ,BR1,BR2,Lietal14,Vo} does not depend on the initial datum (let us note here that, while the KPP assumption implies that $c^* = 2 \sqrt{g'(0)}$, the converse does not hold~\cite{hrothe}).

Nonetheless, a striking feature of Theorem~\ref{th:regime} is the fact that when $c\in [2\sqrt{g'(0)}, c^*)$, whether the solution persists or not depends on the initial datum, see statement $(ii)$ above. This is in sharp contrast with the so-called `hair-trigger effect' for the classical homogeneous monostable equation~\eqref{problemhomo1}, whose solution spreads as soon as the initial datum is non trivial and nonnegative. This result also highlights qualitative differences with the KPP framework of~\cite{BDNZ,BR1,BR2,Lietal14,Vo}, where the persistence of the population depends only on the value of $c$. The biological implication is that under the combination of a weak Allee effect and a shifting climate, the size of the initial population becomes crucial for the survival of the species.

This new behaviour can be understood from the appearance, in the range of speeds $c \in [2 \sqrt{g'(0)},c^*)$, of intermediate stationary solutions between 0 and $p_+^c$. It turns out that, although the equation~\eqref{problemmonobis} is of the monostable type in the half line $\{z > 0\}$, it shares some features with the usual bistable case as the trivial state 0 becomes stable with respect to some small enough compactly supported pertubations. This leads to the following dichotomy, or sharp threshold between extinction and spreading, in the same spirit as the results of~\cite{DM,MZ,polacik,zlatos} in the spatially homogeneous framework:
\begin{Th}\label{th:regime_sharp}
Assume that $2 \sqrt{g'(0)} < c^*$ and choose some $c \in [2 \sqrt{g'(0)},c^*)$. Then for any strictly ordered and continuous (in the $L^1 (\R)$-topology) family $(u_{0,\sigma})_{\sigma > 0}$ of initial data satisfying the same assumptions than in Theorem \ref{th:regime}, there exists some $\sigma^* \in [0,+\infty]$ such that spreading occurs for $\sigma > \sigma^*$, extinction for $\sigma < \sigma^*$, and grounding for $\sigma= \sigma^*$ whenever $\sigma^* \in (0,+\infty)$.
\end{Th}
Theorem~\ref{th:regime_sharp} further highlights that spreading and extinction are the two reasonable outcomes. Indeed, the only other possibility is grounding, and as can be seen when looking at an ordered family of initial data~$(u_\sigma)_\sigma$, this may only occur for a critical choice of the parameter $\sigma^* \in (0,+\infty)$.

Note that this threshold phenomenon strongly relies on our choice of compactly supported initial data. Indeed, one may for instance check that for any initial datum which does not decay to 0 as $z \to +\infty$, spreading necessarily occurs. A similar threshold phenomenon was also exhibited in~\cite{MZ} in the homogeneous but degenerate monostable framework, namely equation~\eqref{problemhomo1} where $g$ satisfies~\eqref{monostable} except that $g'(0)=0$.\\

The above two theorems describe what happens in the different range speeds depending on the initial data. Let us now adopt a different approach where the initial datum is fixed and the speed~$c$ varies. Our last theorem writes as follows:
\begin{Th}\label{th:critspeed}
For all nonnegative, non trivial, bounded and compactly supported function $u_0$, there exists $c (u_0) \in [2 \sqrt{g'(0)},c^*]$ such that the following three statements hold true.
\begin{enumerate}[$(i)$]
\item For all $c < c (u_0)$, spreading occurs.
\item For all $c > c (u_0)$, extinction occurs.
\item If $c = c (u_0)$:
\begin{enumerate}[$(a)$]
\item if $c(u_0) > 2 \sqrt{g'(0)}$, grounding occurs and $ 2 \sqrt{g'(0)}<c(u_0)<c^*$;
\item if $c(u_0)=2\sqrt{g'(0)}<c^*$, there may be either grounding or extinction;
\item if $c = c(u_0) = 2 \sqrt{g'(0)} = c^*$, extinction occurs.
\end{enumerate}
\end{enumerate}
\end{Th}
Theorem~\ref{th:critspeed} shows that there still exists, in the general monostable framework, a threshold forced speed below which spreading occurs and above which the solution goes extinct. As mentioned above, the existence of such a threshold for persistence was already known in the KPP framework. However, Theorem~\ref{th:critspeed} together with Theorem~\ref{th:regime} clearly imply that $c (u_0)$ depends in a non trivial way on the initial datum as soon as $c^* > 2 \sqrt{g'(0)}$. 
From an ecological point of view, this means that the persistence of the population is determined by the value of the climate shift speed with respect to this threshold. When the per capita growth rate of the population is optimal at zero density (no Allee effect), this threshold speed is independent of the initial datum, whereas in the presence of a weak Allee effect, this threshold depends on the size of the initial datum. More precisely, it follows from Theorem~\ref{th:regime_sharp} that $c (u_0)$ is nondecreasing with respect to the initial condition $u_0$, so that a large population will be less sensitive to climate change in the sense that it can keep pace with a faster shifting habitat than a smaller population.

\bigskip

\begin{Rk}
The $C^{1,r}_{loc}$ regularity of $g$ plays an important role in our main results whenever $c = 2 \sqrt{g'(0)}$. Indeed assume for instance that $\displaystyle g(u) \geq g'(0) u \left(1 - \frac{1}{\ln u} \right)$ in a neighbourhood of $0$. From the phase plane analysis of the ODE $$p'' + 2 \sqrt{g'(0)} p' + f(z,p)=0,$$ and proceeding as in Section~\ref{sec:stationarysol}, one may check that there does not exist any nonnegative and bounded stationary solution other than 0 and the invasion state $p_+^c$. While we do not study such a case here, this leads us to formally expect that spreading then occurs for all initial data when $c = 2 \sqrt{g'(0)}$.
\end{Rk}
\bigskip

\textbf{\underline{Organisation of the paper}}
\bigskip

In Section \ref{sec:stationarysol} we study the stationary problem in the moving frame, i.e. the stationary solutions of~\eqref{problemmonobis}. In Theorem~\ref{th:stationary} we classify all the different stationary solutions, and then in Section~\ref{sec:expdecay} we describe the decay of ground states to 0 at infinity. In Section~\ref{sec:regime} we examine the different speed ranges and prove Theorem~\ref{th:regime}. In particular we show that the hair-trigger effect does not hold in the intermediate speed range $c \in [2 \sqrt{g'(0)},c^*)$, see Theorem~\ref{thm:interm_0}. In Section~\ref{sec:convstatsol}, we prove the convergence of the solution of~\eqref{problemmonobis} to a stationary solution as time goes to infinity. Furthermore, combining the uniform in time exponential estimates of Section~\ref{sec:energy} and the decay properties of ground states from Section~\ref{sec:expdecay}, we show that the limiting stationary solution may only be 0, the invasion state or the critical ground state. In the last Section~\ref{sec:sharpf}, we deal with the sharp threshold phenomena and prove both Theorems~\ref{th:regime_sharp} and~\ref{th:critspeed}. 
Lastly, we include in an Appendix~\ref{A:number_zero} some results on the so-called `zero number argument' from~\cite{A,DM}, which we use in Section~\ref{sec:convstatsol} and in particular for the large-time convergence to a stationary solution.

\section{Stationary solutions in the moving frame}\label{sec:stationarysol}
In this section, we are interested in the stationary solutions of the equation~\eqref{problemmonobis}, which may also be seen as travelling wave solutions of the original problem in the sense that they are also entire solutions of \eqref{problemmono} moving with constant speed and profile. We will prove the following theorem which classifies all the stationary solutions of~\eqref{problemmonobis}, making more rigorous the discussion in the beginning of Section~\ref{sec:main_res}.

\begin{Th}\label{th:stationary}
All the positive and bounded stationary solutions of \eqref{problemmonobis} can be classified as follows.
\begin{enumerate}[(i)]
\item For any $c \geq 0$, there exists a maximal positive solution $p_+^c$, which we call \textbf{invasion state} and satisfies
$$0 = \lim_{z \to -\infty} p_+^c < \lim_{z \to +\infty} p_+^c =1 \ , \quad \partial_z p_+^c (z) >0.$$
\item If $c \geq 2 \sqrt{g'(0)} $, there exists $\alpha^*_c \in (0, p_+^c (0)]$ and a family of positive stationary solutions $(p_\alpha)_{0 < \alpha < \alpha^*_c} $, that will be called \textbf{ground states}, which satisfy:
\be\label{palphaproperties}
\lim_{z \to \pm \infty} p_\alpha = 0 \ , \quad p_\alpha < p_+^c \ , \quad p_\alpha (0 ) = \alpha ,\ee
and $c\mapsto\alpha^*_c$ is nondecreasing. Moreover:
\begin{enumerate}[(a)]
\item if $c < c^*$, then $\alpha^*_c < p_+^c (0)$ and there exists a positive stationary solution $p_{\alpha^*_c}$ satisfying the same properties~\eqref{palphaproperties}, which we call the \textbf{critical ground state};
\item on the other hand, if $c \geq c^*$, then $\alpha^*_c = p_+^c (0)$ and $p_\alpha \to p_+^c$ locally uniformly as $\alpha \to p_+^c (0)$.
\end{enumerate}
\end{enumerate}
There exists no other positive and bounded stationary solution of \eqref{problemmonobis} than the ones defined above.
\end{Th}
This theorem already highlights three different situations, depending on the forcing/shifting speed $c \geq 0$. First, if $0 \leq c < 2 \sqrt{g'(0)}$, then there exists a unique positive and bounded stationary solution, which intuitively means that the equation retains its monostable feature. However, as soon as $c \geq 2\sqrt{g'(0)}$, some intermediate stationary solutions emerge in a neighbourhood of the trivial steady state 0, which thus becomes stable with respect to small enough perturbations. This leads to the loss of the hair-trigger effect, and the more complex dynamics stated in our Theorems~\ref{th:regime},~\ref{th:regime_sharp} and~\ref{th:critspeed}. Furthermore, when $c\geq c^*$, these intermediate stationary solutions even form some sort of foliation from 0 to the maximal stationary solution~$p^c_+$, which completely prevents the propagation, at least for compactly supported initial data. One can look at Figure~\ref{phaseportrait3} for an illustration of the different stationary solutions in the phase plane $(p,p')$ depending on the value of $c$.

\subsection{Proof of Theorem~\ref{th:stationary}}\label{sec:stationarysol_sub1}
Note that any stationary solution $p (z)$ of \eqref{problemmonobis} satisfies the second-order ODE
\begin{equation}\label{ODE1}
p'' + cp' + f(z,p)=0.
\end{equation}
This equation is homogeneous on each half interval $\{ z < 0\}$ and $\{z > 0\}$ of the domain, thus we will construct stationary solutions by `glueing' phase portraits as in Berestycki et al \cite{BDNZ}.
\begin{figure}[h]
\centering
\includegraphics[scale=0.29]{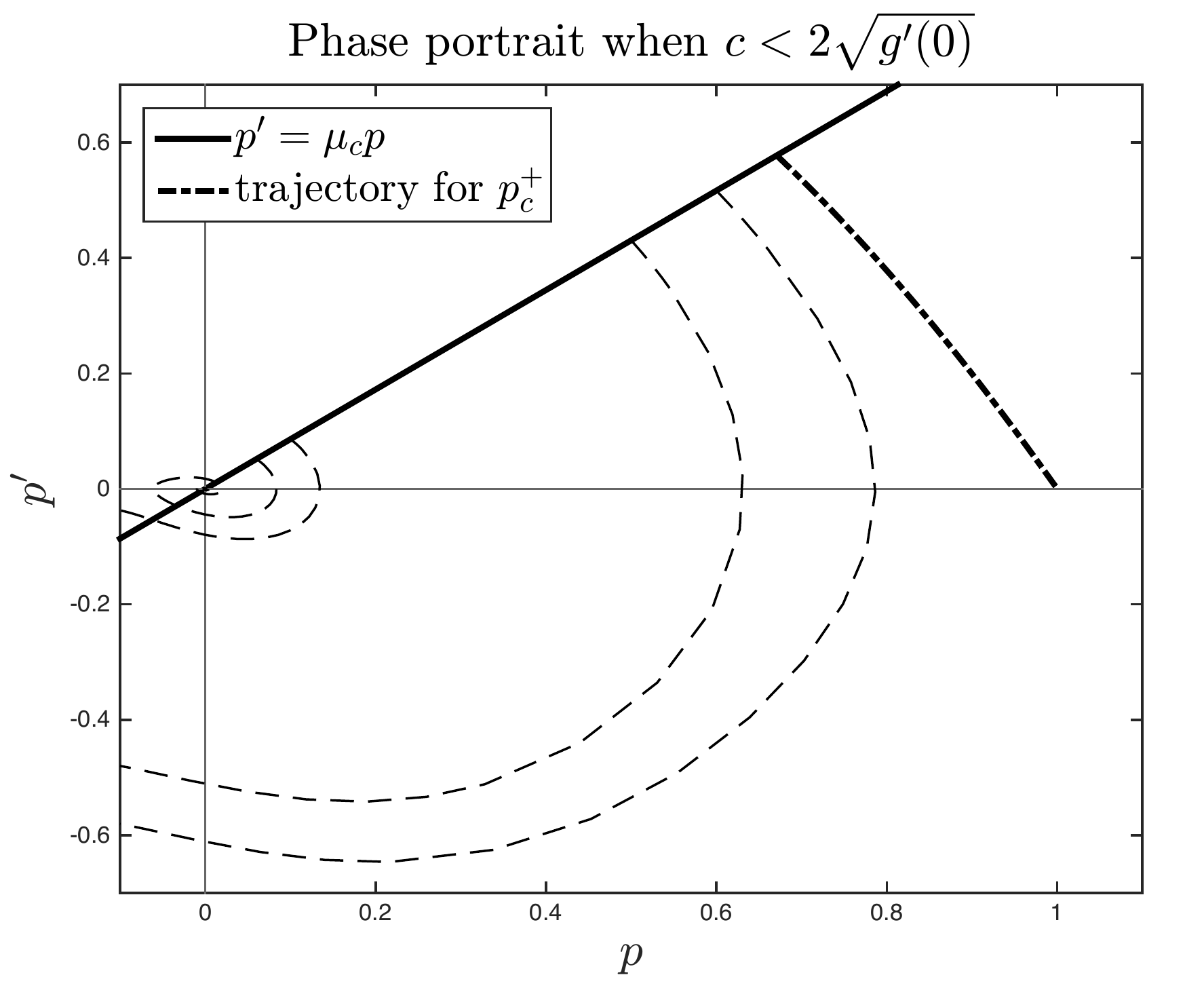}
\includegraphics[scale=0.29]{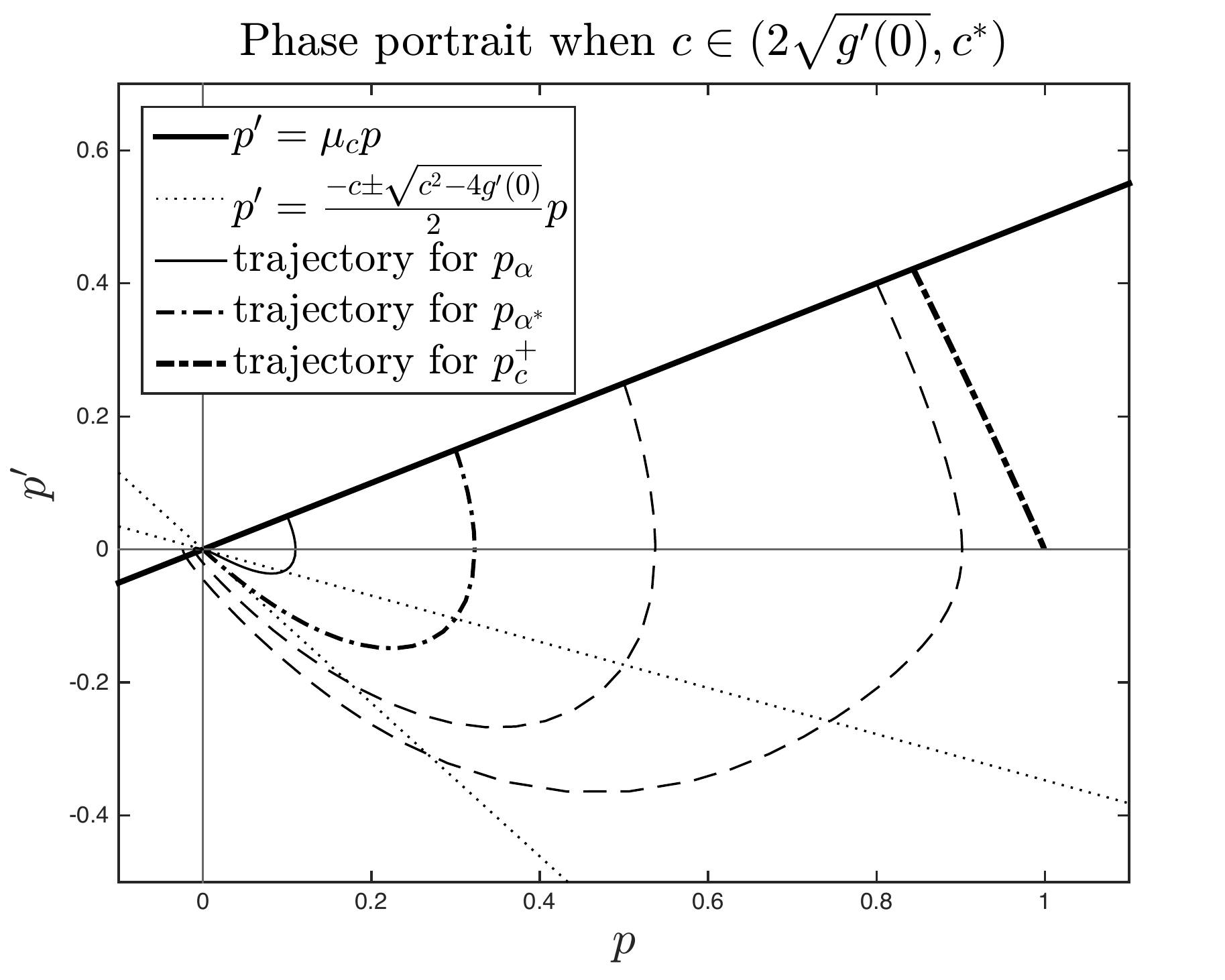}
\includegraphics[scale=0.29]{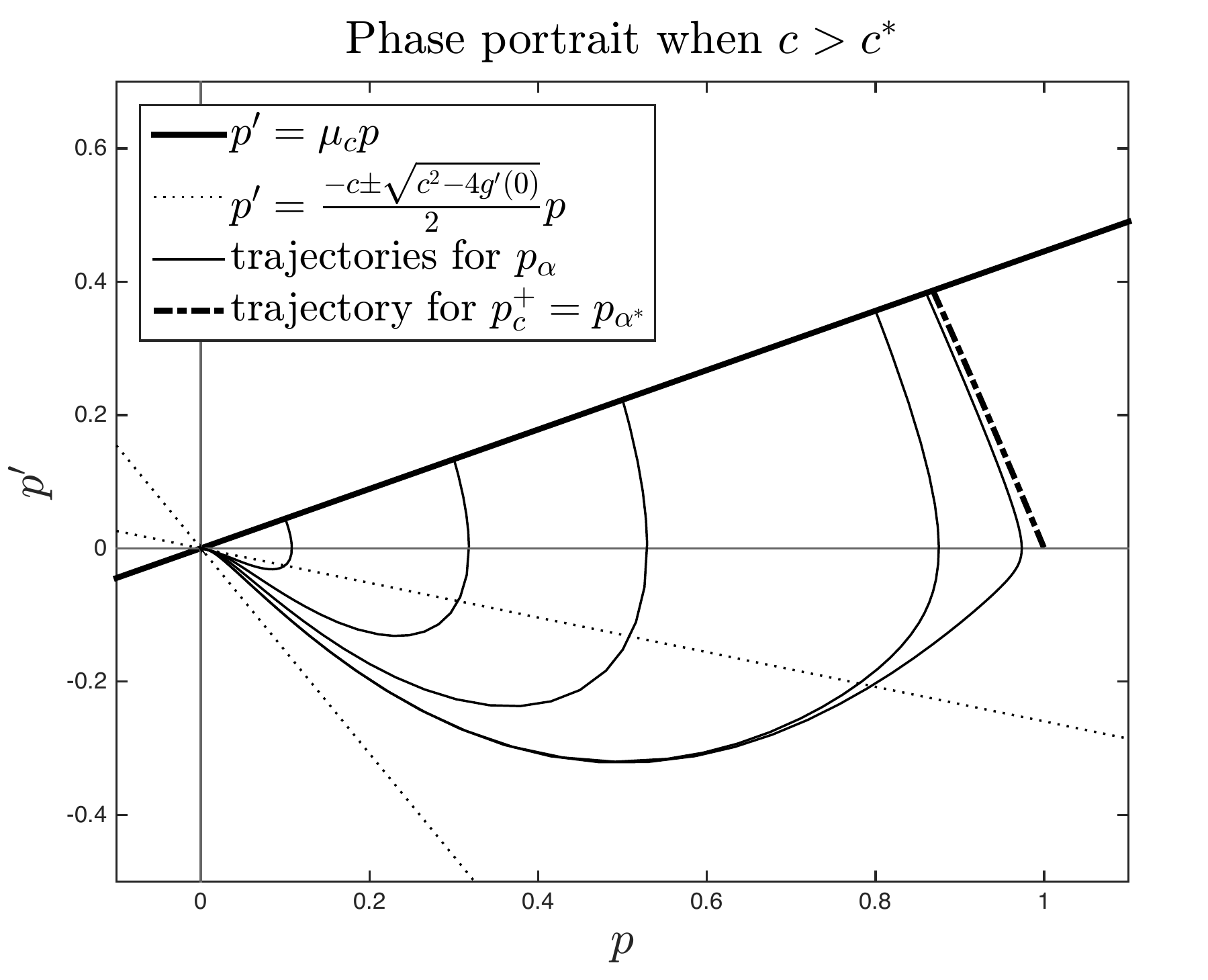}
\caption{Phase portrait of the discontinuous ODE~\eqref{ODE1} for different values of $c$. All the trajectories start on the line $p'=\mu_c p$ before entering the phase plane of the homogeneous equation~\eqref{ODE2}. The trajectories in dashed lines are the ones that cross the $p'$-axis, the one in bold dashed-dotted line is the critical ground state $p_{\alpha^*_c}$ and the one in bold dashed line is the maximal solution $p_+^c$. These simulations were conducted using Matlab and the function $g(s)=4s(1-s)(\sin(s )+0.1)$.}\label{phaseportrait3}
\end{figure}
\begin{proof}[Proof of Theorem \ref{th:stationary}]
For $z \leq 0$, a stationary solution of \eqref{problemmonobis} satisfies
$$p'' + cp' - p =0.$$
As we are only interested in positive and bounded stationary solutions, it immediately follows that
\begin{equation}\label{eq:left}
\forall z \leq 0 \ : \quad p (z)= p (0) \times e^{\mu_c z},
\end{equation}
where $\mu_c$ is defined as
$$\mu_c := \frac{-c + \sqrt{c^2 + 4}}{2} >0.$$
Let us now `glue' this with a solution of the homogeneous monostable equation
\be\label{ODE2}
p'' + cp' + g(p)=0, 
\ee
on $[0,+\infty)$. The identity~\eqref{eq:left} means that we are now looking at trajectories of \eqref{ODE2} starting from some point on the half line $\{ p' = \mu_c p \, | \ p >0  \}$, as illustrated in Figure~\ref{phaseportrait3}. We argue in the phase plane of \eqref{ODE2} and construct in the three steps below all the stationary solutions of \eqref{problemmonobis}. We refer to Aronson and Weinberger~\cite{AW} for more complete and detailed arguments on the phase plane analysis of \eqref{ODE2}.\\

\textit{Step 1: existence of a maximal solution $p^c_+$.} We start by proving the first point in Theorem \ref{th:stationary}, and take any $c\geq0$. It is easy to check that the steady state $(p=1,p'=0)$ is a saddle point. In particular, it admits a one dimensional stable manifold and we can consider the unique trajectory which converges to $(1,0)$ from the upper phase plane $\{ p' > 0\}$. This trajectory clearly crosses the half line $\{ p' = \mu_c p \, | \ p >0  \}$ and, up to some shift, we can assume without loss of generality that the associated solution $p$ satisfies $p'(0)=\mu_c p(0)$ and converges monotonically to $1$ as $z \to +\infty$. Glueing it with the exponential \eqref{eq:left}, we obtain a positive and bounded stationary solution $p_+^c$ as in Theorem~\ref{th:stationary}.

\textit{Step 2: non existence of positive bounded solutions $p$ with $p(0) > p_+^c (0)$.}
Now let $p$ be any stationary solution that satisfy~\eqref{eq:left} with $p(0) =\alpha > p_+^c (0) >0$, and consider the trajectory in the phase plane of \eqref{ODE2} starting from $(\alpha, \mu_c \alpha )$. Clearly it lies above the trajectory of $p_+^c$ and, therefore, it crosses the vertical line $\{ p = 1\}$ above $(1,0)$. Then, as $g(s) < 0$ for all $s > 1$, it is straightforward that the trajectory goes to infinity. In other words, any stationary solution of~\eqref{ODE1} satisfying $p(0) > p_+^c (0)$ is unbounded. So there does not exist any bounded and nonnegative solution of the stationary problem \eqref{ODE1} that lies above $p_+^c$. 

\textit{Step 3: existence of non trivial intermediate solutions (ground states) if and only if $c \geq 2 \sqrt{g'(0)}$.} Next let $p$ be any stationary solution satisfying~\eqref{eq:left} with $p(0) = \alpha \in (0, p_+^c (0))$. Let us first note that the trajectory of $p$ in the phase plane of \eqref{ODE2}, which starts from $(\alpha,\mu_c \alpha)$, crosses the horizontal axis $\{ p' =0\}$ and enters the $\{p' < 0  \mbox{ and } 0 < p < 1\}$ part of the phase plane. As $g(s) >0$ for all $0< s < 1$, clearly it cannot cross back the horizontal axis. Thus the trajectory may only leave the set $\{p' < 0  \mbox{ and } 0 < p < 1\}$ by crossing the vertical axis $\{ p=0\}$ below the origin and, if it does not, then it converges to the equilibrium $(0,0)$. 

We now divide the proof into three parts depending on the value of $c$. If $c < 2 \sqrt{g'(0)}$, it is rather straightforward that the linearized operator around the steady state $(p=0,p'=0)$ only admits two complex eigenvalues. Thus, any stationary solution~$p$ which is nonnegative for $z \leq 0$ and satisfies $0< p(0) < p_+^c (0)$ changes sign in $\R_+$. Putting this together with the step~2 above, we are now able to conclude that there exists no positive and bounded stationary solution other than $p_+^c$.

Next, we assume that $c \geq c^*$. Let us recall that $c^* >0$ is the smallest $c \in \R$ such that there exists a trajectory in the phase plane, connecting the two steady states $(p=1,p'=0)$ and $(p=0,p'=0)$, whose associated solution of \eqref{ODE2} is exactly the travelling wave $V_c$ with speed~$c$~\cite{AW}. In particular, if $c \geq c^*$, then clearly the trajectory starting from $(\alpha, \mu_c \alpha)$ with $\alpha \in (0,p_+^c (0))$ must lie above that of $V_c$, hence does not cross the vertical axis and converges to~0 as $z \to +\infty$. For each $0 < p(0) < p_+^c (0)$ we obtain a (unique) positive and bounded stationary solution, as announced in Theorem~\ref{th:stationary}$(b)$.

It now only remains to consider the case $2 \sqrt{g'(0)} \leq c < c^*$. In this case, the steady state $(0,0)$ is a node, namely trajectories locally converge to $(0,0)$ along the lines of slope $- \frac{c \pm \sqrt{c^2 - 4 g'(0)}}{2}$. Note that, in the critical case $c = 2 \sqrt{g'(0)}$, this crucially relies on the $C^{1,r}$ regularity of $g$. Furthermore, it can be shown that there exists a trajectory converging to $(0,0)$ from the $\{p' < 0 < p\}$ part of the phase plane, which is extremal in the sense that any trajectory lying below it must cross the vertical axis $\{ p =0 \}$ before converging to $(0,0)$. We temporarily denote by $p^*$ the solution of \eqref{ODE2} associated with this extremal trajectory ($p^*$ is defined up to any shift). Following backward this extremal trajectory, it is clear that it must leave the $\{p' < 0 \mbox{ and } 0 < p < 1\}$ part of the phase plane, either through the horizontal axis between $(0,0)$ and $(1,0)$, or through the vertical line $\{p=1\}$. The latter contradicts the fact that, since $c < c^*$, the trajectory originating from the unstable manifold of $(1,0)$ may not converge to $(0,0)$ without crossing the vertical axis. 
Moving further back on the trajectory, it becomes clear that it also intersects the half line $\{ p' = \mu_c p \, | \  p >0 \}$ below the point $(p_+^c (0), p_+^c  \, ' (0))$. Denoting by $\alpha^*_c \in (0,p_+^c (0))$ the horizontal coordinate of this intersection, it is now straightforward that, for any $ \alpha \leq \alpha^*_c$, the stationary solution $p_\alpha$ of \eqref{ODE1} satisfying $p(0)=\alpha$ remains positive and converges to~0 as $z \to +\infty$, while for $\alpha^*_c < \alpha < p_+^c (0)$ it changes sign. Glueing all the trajectories which do not cross the vertical axis with the exponentials~\eqref{eq:left}, we obtain all the stationary solutions described in Theorem~\ref{th:stationary}$(a)$, and by construction there exists no other positive and bounded stationary solution of \eqref{problemmonobis}.\\

We are now in a position to conclude the proof of Theorem~\ref{th:stationary}. We first show that $c\mapsto \alpha^*_c$ is nondecreasing, and consider two different speeds $2\sqrt{g'(0)} \leq c_1<c_2$. Choose then any $\alpha\in (0,\alpha^*_{c_1})$. From a phase plane analysis and using the fact that $c\mapsto\mu_c$ is decreasing, one can prove that the trajectory of \eqref{ODE2} with $c=c_1$ and starting at $(\alpha^*_{c_1},\mu_{c_1}\alpha^*_{c_1})$ lies above (respectively below) the trajectory of \eqref{ODE2} with $c=c_2$ and starting at $(\alpha,\mu_{c_2}\alpha)$ in the $\{p>0,\:p'>0\}$ part of the phase plane (respectively the $\{p'<0<p\}$ part of the phase plane). It is then straighforward that the trajectory of \eqref{ODE2} with $c=c_2$ starting at $(\alpha,\mu_c \alpha)$ converges to $(0,0)$ without crossing the vertical axis. It follows that $\alpha^*_{c_2} \geq \alpha$ and, as $\alpha$ could be chosen arbitrarily close to $\alpha^*_{c_1}$, we conclude that $\alpha^*_{c_2} \geq \alpha^*_{c_1}$.

It only remains to check that the intermediate stationary solutions $p_\alpha$ constructed in the step 3 above satisfy the inequality $p_\alpha < p_+^c$ on the whole line. For $z \leq 0$, it immediately follows from the fact that all stationary solutions are exponential, see \eqref{eq:left}. Moreover, it is straightforward from our construction in the phase plane (see also Figure~\ref{phaseportrait3}) that there does not exist a $z >0$ such that $p_\alpha (z) = p_+^c(z)$ and $p'_\alpha  (z) \geq p_+^{c} \, ' (z)$, which immediately implies that $p_\alpha < p_+^c$ for~$z \geq 0$. In particular, if $c \geq c^*$ and $\alpha \to p_+^c (0)$, then $p_\alpha$ converges locally uniformly to a stationary solution $p_\infty \leq p_+^c$ such that $p_\infty (0) = p_+^c (0)$, hence $p_\infty \equiv p_+^c$ by the strong maximum principle. Note, however, that this convergence could also be deducted directly from the phase plane. This ends the proof of Theorem~\ref{th:stationary}.
\end{proof}

\begin{Rk}
We mention here that, in~\cite{AW}, the minimal wave speed $c^*$ is obtained as the smallest $c$ such that the extremal trajectory defined above crosses the vertical line $\{p =1 \}$ below $(1,0)$. In particular, even though we treated above the cases $2 \sqrt{g'(0)} \leq c < c^*$ and $c \geq c^*$ separately, one may see that both arguments actually follow from the same idea.
\end{Rk}

\subsection{Exponential decay of ground states}\label{sec:expdecay}

Theorem~\ref{th:stationary} states that, when $c \geq 2 \sqrt{g'(0)}$, there exist infinitely many ground states between 0 and~$p_+^c$. In order to understand the dynamics of the time evolution problem, one may want further insight on these intermediate steady states and, for instance, may wonder whether ground states are ordered or not. It follows from Theorem~\ref{th:stationary} and Proposition~\ref{prop:exp_decay} below that they cannot since the critical ground state (when it exists) satisfies both $p_{\alpha^*_c} (0) > p_\alpha (0)$ and $p_{\alpha^*_c} (z) < p_\alpha (z)$ for all $z$ large enough, for any $\alpha \in (0,\alpha^*_c)$. 

Furthermore, it turns out that the way ground states decay as $z \to +\infty$ plays an essential role in determining whether they may appear in the large-time behaviour of solutions under our choice of compactly supported initial data.
\begin{Prop}\label{prop:exp_decay}
Assume that $c \geq 2\sqrt{g'(0)}$ and define
$$\lambda_\pm (c) := \frac{c \pm \sqrt{c^2 - 4 g'(0)}}{2}.$$
\begin{enumerate}[(i)]
\item If $c > 2 \sqrt{g'(0)}$, then for any $0< \alpha < \alpha^*_c$, the ground state $p_\alpha$ satisfies
$$ p_\alpha (z) = A e^{-\lambda_- (c) z} (1 + O (e^{-\delta z}))  \mbox{ as } z \to +\infty , $$
while (provided that $c < c^*$), the ground state $p_{\alpha^*_c}$ satisfies
$$p_{\alpha^*_c}  (z) = A e^{-\lambda_+ (c) z} (1 + O (e^{-\delta z}) ) \mbox{ as } z \to +\infty ,$$
where in both cases $A$ and $\delta$ are positive constants.

\item For $c = 2 \sqrt{g'(0)}$ and any $0< \alpha \leq \alpha^*_c$, there exists $\delta >0$, $A \geq 0$ and $B\in \R$ such that
$$p_\alpha (z) = (Az + B) e^{-\frac{\sqrt{g'(0)}}{2} z} + O(e^{-(\frac{\sqrt{g'(0)}}{2}+\delta ) z} );$$
more precisely, $A >0$ if $\alpha< \alpha^*_c$, while $B> A=0$ if $\alpha = \alpha^*_c$.
\end{enumerate}
\end{Prop}
\begin{figure}
\centering
\includegraphics[scale=0.4]{PhasePortrait0node.pdf}\quad
\includegraphics[scale=0.4]{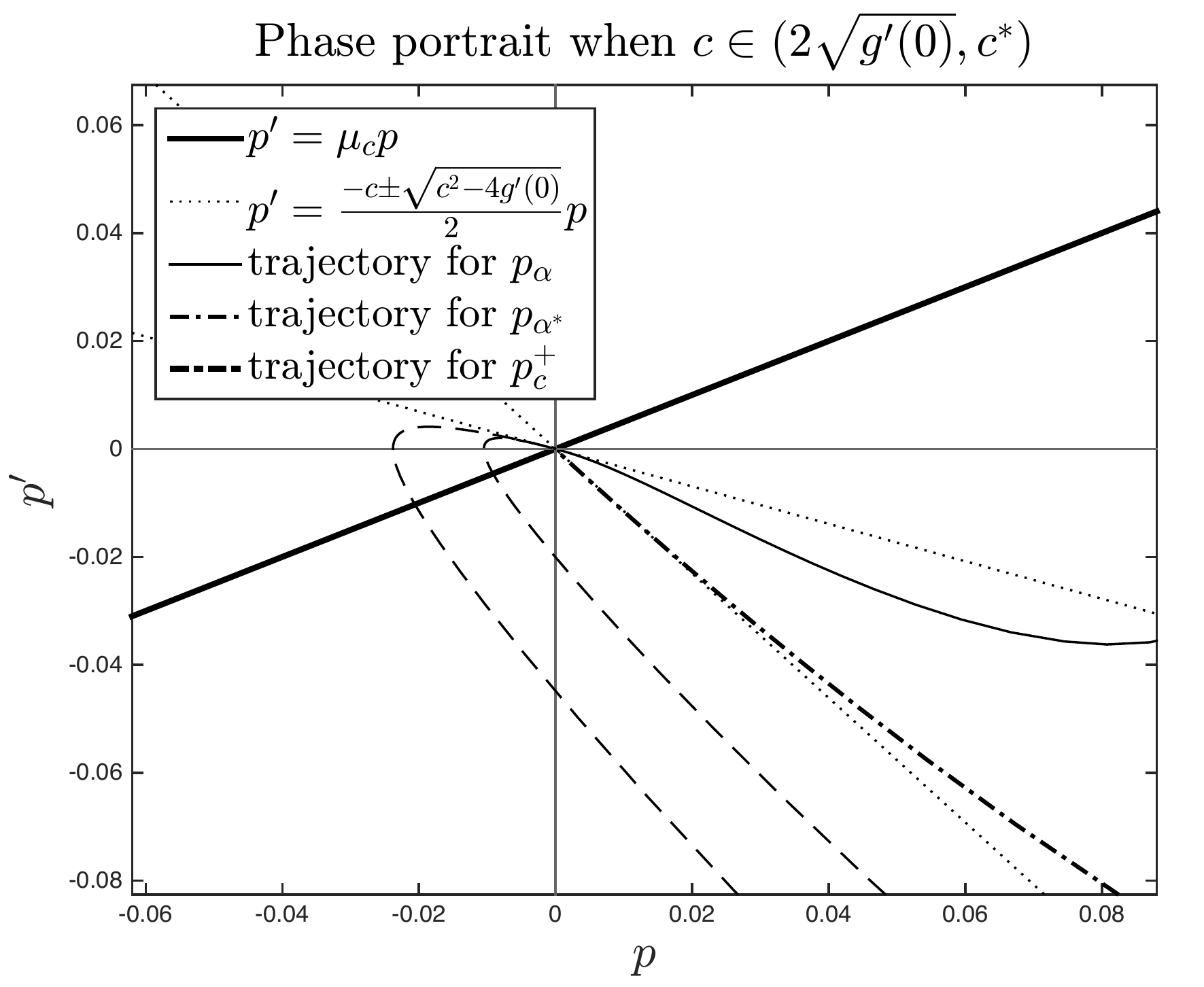}
\caption{Phase portrait of the steady states of problem \eqref{problemmonobis} in the moving frame when $2\sqrt{g'(0)}<c<c^*$. The figure on the right is a zoom of the left figure. One can notice, as stated in Proposition \ref{prop:exp_decay}, that for all $\alpha<\alpha^*$, the stationary solutions $p_\alpha$ (defined in Theorem \ref{th:stationary}) converge to 0 asymptotically to the line $p'=-\lambda_- p$, whereas $p_{\alpha^*}$ converges to $0$ asymptotically to the line $p'=-\lambda_+ p$.}\label{phaseportrait0node}
\end{figure}
\begin{figure}
\centering
\includegraphics[scale=0.4]{PhasePortrait0nodelargec.pdf}\quad
\includegraphics[scale=0.4]{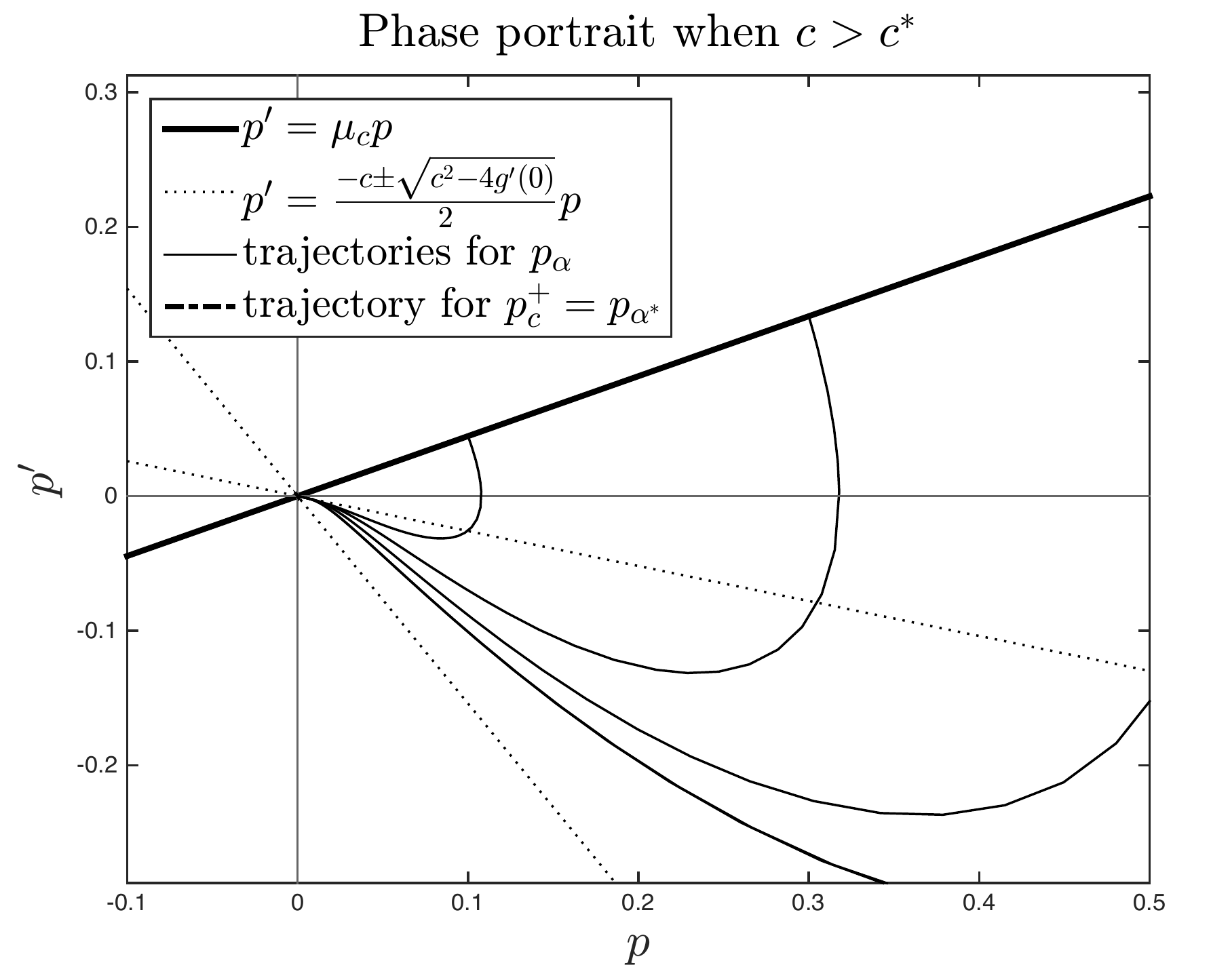}
\caption{Phase portrait of the steady states of problem \eqref{problemmonobis} in the moving frame when $c>c^*$. The figure on the right is a zoom of the left figure. Here the stationary solutions $p_\alpha$ all converge to 0 asymptotically to the line $p'=-\lambda_- p$.}\label{phaseportrait0nodelargec}
\end{figure}
\begin{proof}[Proof of Proposition~\ref{prop:exp_decay}]
We briefly sketch the proof, which follows from the construction of ground states above, and from standard phase plane analysis of the homogeneous monostable equation~\eqref{ODE2}. We also refer to the Figures~\ref{phaseportrait0node} and~\ref{phaseportrait0nodelargec} for an illustration of the argument.

When $c > 2 \sqrt{g'(0)}$, the linearized problem around the steady state $(0,0)$ admits two eigenvectors $(1,-\lambda_\pm (c))$. Noting that $|\lambda_- (c)| < |\lambda_+ (c)|$, it is well-known (see for instance~\cite{cl}), that all trajectories converging to $(0,0)$ do so along the line $\{p' = -\lambda_- (c) p \}$, except for a single trajectory which goes through $(0,0)$ along the line $\{p' = - \lambda_+ (c) p\}$. In particular, this latter trajectory lies below all the others which converge to $(0,0)$ from the $\{p' < 0 < p\}$ part of the phase plane without crossing the vertical axis. It is also what we refered to earlier as the extremal trajectory which, by construction (see the proof of Theorem~\ref{th:stationary}), coincides with the trajectory of the critical ground state $p_{\alpha^*_c} (z)$ when $c < c^*$. Note that, if $c \geq c^*$, the extremal trajectory lies below the trajectory of the travelling wave with speed $c$ (with which it actualy coincides if $c=c^* > 2 \sqrt{g'(0)}$). Therefore, it does not cross the horizontal axis between $(0,0)$ and $(1,0)$ and thus does not coincide with any positive and bounded stationary solution of \eqref{ODE1}. 

It follows from the discussion above that, if $c > 2 \sqrt{g'(0)}$, then $p_{\alpha^*_c}$ converges to $0$ as $z \to +\infty$ with the exponential rate~$-\lambda_+$, while any other ground state $p_{\alpha}$ converges to 0 with the exponential rate~$-\lambda_-$. The more accurate asymptotics stated in Proposition~\ref{prop:exp_decay} are a classical consequence of the $C^{1,r}_{loc}$-regularity of~$g$~\cite{cl}.

The critical case $c=2\sqrt{g'(0)}$ follows from a similar argument. While the linearized problem around~$(0,0)$ admits a unique eigenvector and, therefore, all trajectories converge to $(0,0)$ along the same line $\{ p' = - \frac{\sqrt{g'(0)}}{2} p\}$, it is still known that trajectories behave in a similar fashion as in the linear problem thanks to the $C^{1,r}_{loc}$-regularity of $g$. More precisely, there exists an extremal trajectory which decays to 0 with the asymptotics $B e^{-\frac{\sqrt{g'(0)}}{2}}$ with $B>0$, while all the other trajectories decay to 0 with the slightly slower asymptotics $(Az+B) e^{-\frac{\sqrt{g'(0)}}{2}}$ where $A >0$ and $B \in \R$. The end of the proof of Proposition~\ref{prop:exp_decay} is again a simple consequence of the construction of $p_{\alpha^*_c}$ as the unique stationary solution of \eqref{problemmonobis} associated with the extremal trajectory.
\end{proof}

\section{Large-time behaviour: spreading and extinction}\label{sec:regime}

In this section we prove most of Theorem \ref{th:regime} about the asymptotic behaviour of the solution $u(t,z)$ of problem \eqref{problemmonobis}, and in particular we show how the hair-trigger effect disappears for intermediate speeds. The results presented in this section mostly rely on direct comparison methods. However, we will leave partly open the critical case $c = c^*$, which will be dealt with later on (see Theorem~\ref{th:nospread} and Proposition~\ref{prop:conv}).\\

Before starting the proof, we state some lemma which will be used extensively in the following sections:
\begin{Lemma}\label{lem:sup_cv}
Let $u_0$ be any nonnegative, bounded and compactly supported initial datum, and $u(t,z)$ be the associated solution of \eqref{problemmonobis}. Then
$$\lim_{t\to +\infty} \ \sup_{z \in \R} \left(  u (t,z) - p_+^c (z) \right) = 0.$$
\end{Lemma}
\begin{proof}
Note that any constant $M>1$ is a supersolution of \eqref{problemmonobis}. Thus, the solution $\overline{u}_M$ of \eqref{problemmonobis} with initial datum $M$ is decreasing in time and, by standard parabolic estimates, it converges locally uniformly to a stationary solution. As $p_+^c \leq M$ is the largest positive and bounded stationary solution, it easily follows that the solution actually converges to $p_+^c$. Choosing $M$ large enough so that $u_0 (z) \leq M$ for all $z \in \R$ and by the parabolic comparison principle, we can already conclude that
\begin{equation}\label{newlem_loc}
\forall  Z>0, \quad \limsup_{t \to +\infty} \ \sup_{|z|\leq Z} (u(t,z)-p_+^c (z))  \leq 0.
\end{equation}
Let us now improve the above estimate. Recall that
$$\mu_c = \frac{-c + \sqrt{c^2 +4}}{2} >0$$
is such that the exponential $A e^{\mu_c z}$ satisfies \eqref{problemmonobis} for all $z \leq 0$ and any $A>0$. One can then choose $A>0$ such that $u_0 (z) \leq Ae^{\mu_c z}$ for all $z \leq 0$, and such that $u (t,z=0) \leq  A$ for all $t \geq 0$ thanks to the fact that $u(t,z=0) \leq \overline{u}_M (t,z=0) \leq M$ for all $t >0$. By the parabolic maximum principle, we can infer that 
\begin{equation}\label{newlem_1}
u(t,z) \leq A e^{\mu_c z}
\end{equation}
for all $t \geq 0$ and $z \leq 0$.

On the other hand, let us now denote by $\overline{v}_M$ the solution of the ordinary differential equation $\overline{v}_M ' (t) = g(\overline{v}_M)$ with $\overline{v}_M (0) = M.$ Clearly it converges to 1 as $t \to +\infty$ and, since it is a supersolution of \eqref{problemmonobis}, we get that
\begin{equation}\label{newlem_2}
\limsup_{t\to +\infty} \ \sup_{z \in \R} u (t,z) \leq 1.
\end{equation}
We are now in a position to conlude the proof. Choose any $\varepsilon >0$. Then let $Z>0$ be large enough such that 
\begin{equation*}
\left\{
\begin{array}{l}
\displaystyle Ae^{-\mu_c Z} \leq \varepsilon,\vspace{3pt}\\ 
\displaystyle  p_+^c (z) \geq 1 - \varepsilon \mbox{ for all } z \geq Z.
\end{array}
\right.
\end{equation*}
This implies, together with~\eqref{newlem_1} and \eqref{newlem_2}, that 
$$\limsup_{t \to +\infty} \sup_{|z| \geq Z} (u(t,z)-p_+^c (z))  \leq \varepsilon .$$
Putting this together with~\eqref{newlem_loc}, we conclude that
$$\limsup_{t \to +\infty } \ \sup_{z \in \R} (u(t,z)-p_+^c (z)) \leq \varepsilon .$$
Recalling that $\varepsilon$ can be chosen arbitrary small, and noting that $\liminf_{z \to -\infty} u(t,z) - p_+^c (z) \geq 0$ for all $t>0$, the lemma is proved.
\end{proof}

\subsection{In the case $0 \leq c<2\sqrt{g'(0)}$}

We want to prove that the solution $u$ converges in the moving frame to the maximal positive stationary solution $p_+^c$ of \eqref{problemmonobis}, namely statement $(i)$ of Theorem \ref{th:regime}. In particular, this means that the population manages to survive, and even to expand through the favorable zone. In other words, when the shifting speed of the favorable area is slow enough, the hair-trigger effect is still valid. Statement $(i)$ of Theorem \ref{th:regime} follows from the next theorem.
\begin{Th}\label{th:spread_slow}
Assume that $0\leq c < 2\sqrt{g'(0)}$ and $u_0$ is a non trivial, nonnegative and bounded initial datum.

Then the solution $u (t,z)$ of \eqref{problemmonobis} with initial datum $u_0$ converges locally uniformly to $p_+^c$ as $t$ goes to infinity.
\end{Th}
\begin{proof}
Thanks to Lemma~\ref{lem:sup_cv}, we only need to prove that
$$\liminf_{t \to +\infty} \  u(t,\cdot) \geq p_+^c ,$$
where the limit is understood to be locally uniform with respect to~$z$.

To do so, we exhibit a non trivial but arbitrarily small subsolution of~\eqref{problemmonobis}. We argue in the phase plane of \eqref{ODE2}:
\begin{equation*}
\phi'' + c\phi' + g(\phi)=0.
\end{equation*}
Since $c < 2 \sqrt{g'(0)}$, the eigenvalues of the linearized problem around $(0,0)$ are in $\mathbb{C} \setminus \mathbb{R}$. Therefore, for any $\kappa \in (0,1)$, the trajectory in the phase plane going through the point $(\kappa,0)$ crosses the vertical axis twice. Cutting this trajectory and extending it by $0$, it is then straightforward to construct a nonnegative compactly supported subsolution $\phi_\kappa$ of \eqref{problemmonobis}, which tends to 0 as $\kappa \to 0$ and, up to some shift, whose support is included in $[0,+\infty)$. Note also that the size of the compact support stays bounded as $\kappa \to 0$, as it is given by the non trivial imaginary part of the complex eigenvalues of $(0,0)$.

Letting $\underline{u}$ be the solution of \eqref{problemmonobis} with the initial datum $u_0=\phi_\kappa$ we know that $\partial_t \underline{u}>0$ for all $t>0$. Moreover, using again Lemma~\ref{lem:sup_cv} or the fact that any constant $M>1$ is a supersolution of \eqref{problemmonobis}, the function $\underline{u}$ is bounded uniformly in time. Thus, by standard parabolic estimates, it converges locally uniformly to a nonnegative, non trivial and bounded stationary solution. By uniqueness (see Theorem~\ref{th:stationary}), it is clear that the limit of $\underline{u}$ must be~$p_+^c$.

Now consider $u_0$ a non trivial, nonnegative and bounded initial datum. Then, by the strong maximum principle, the associated solution $u(t,z)$ is positive for any positive time. In particular $u (1,\cdot) >0$ and one can choose $\kappa$ small enough so that, for all $z \in \R$,
$$\phi_\kappa (z) \leq u(1,z).$$
Here we used the fact that the size of the support of $\phi_\kappa$ remains bounded as $\kappa \to 0$. Using the parabolic comparison principle together with Lemma~\ref{lem:sup_cv}, we conclude that
$$p_+^c(z)\leq \underset{t\to+\infty}{\liminf}\:u(t,z)\leq\underset{t\to+\infty}{\limsup}\:u(t,z)\leq p_+^c(z),$$
where the limits are understood to be (at least) locally uniform with respect to $z$.
\end{proof}

\subsection{In the case $c\geq c^*$}\label{sec:large_c}
In this section we prove statement $(iii)$ of Theorem \ref{th:regime} . Note that in the particular case $c=c^*$, here we only prove that spreading does not occur.
\begin{Th}\label{th:nospread}
Let $u_0$ be a non trivial, nonnegative, bounded and compactly supported initial datum. 
\begin{itemize}
\item If $c > c^*$, then the solution $u (t,z)$ of \eqref{problemmonobis}, with initial datum $u_0$, converges uniformly to $0$ as $t \to +\infty$.
\item If $c = c^* \geq 2 \sqrt{g'(0)}$, then the solution $u(t,z)$ does not spread. Furthermore, there exists no time sequence $t_n \to \infty$ such that $u(t_n,z) \to p_+^c (z)$ locally uniformly with respect to $z$.
\end{itemize}
\end{Th}
To conclude that extinction occurs in the critical case is slightly more intricate. This will be performed in the next section by a combination of Theorem~\ref{th:nospread} and some exponential bounds: we refer more precisely to Proposition~\ref{prop:conv} and the subsequent discussion.

\begin{proof}In both cases, the proof relies on the fact that solutions of the homogeneous monostable reaction-diffusion equation
\begin{equation}\label{problemmonohom}
\partial_t v-\partial_{xx}v =g(v), \quad t\in(0,+\infty), \: x\in\R,
\end{equation}
are supersolutions of \eqref{problemmono}. 

Classical results from Aronson-Weinberger \cite{AW} imply that the solution~$v$ of \eqref{problemmonohom} with a (compactly supported) initial datum $u_0$ spreads with speed $c^*$, in the sense that
$$
\left\{
\begin{array}{l}
\forall \: 0 < \nu < c^* , \quad \lim_{t \to +\infty} \sup_{0 \leq x \leq \nu t } | v (t,x) -1 |= 0,\vspace{3pt}\\
\forall \:\nu > c^* , \quad \lim_{t \to +\infty} \sup_{x \geq \nu t } v (t,x) = 0.
\end{array}
\right.
$$
In particular, it would follow by the comparison principle that
\begin{equation}\label{eq:nospread1}
\forall \:\nu > c^* , \quad \lim_{t \to +\infty} \sup_{x \geq \nu t } u (t,x) = 0,
\end{equation}
where in this equation $u(t,x)$ is the solution of \eqref{problemmono} in the non moving frame. In fact, the result of Aronson and Weinberger was restricted to the special case $0 \leq u_0 \leq 1$. Here we will prove~\eqref{eq:nospread1} by a similar argument, which has also been known to extend Aronson and Weinberger's result to the general case $0 \leq u_0$ under our assumption that $g (s)< 0$ for all $s >1$.

First recall from Lemma~\ref{lem:sup_cv} that
\be\label{eq:delta}\delta (t) := \sup_{z \in \R} \left( u(t,z) - p_+^c (z) \right) \to 0, \ee
as $t \to +\infty$. Moreover, letting
$$A := \sup_{0\leq u \leq M} \frac{g(u)}{u} ,$$
it is straightforward that, for any $t \geq 0$ and $z\in \R$, the function
$$\overline{u}_2 (t,z) := \frac{e^{At}}{\sqrt{4 \pi t}} \int_\R u_0 (y) e^{-\frac{|z + ct-y|^2}{4t}} dy $$
is a supersolution of \eqref{problemmonobis} which also satisfies $\overline{u}_2 (0,\cdot) = u_0 (\cdot)$. We can now conclude from the above and the comparison principle that, for all $t \geq 0$ and $z \in \R$,
$$u(t,z) \leq \min \{ p_+^c (z) + \delta (t),  \overline{u}_2(t,z) \}.$$
Next, we define $\eta >0$ small enough so that
$$\forall s \in (1-\eta,1+\eta), \quad g' (s) \leq \frac{g'(1)}{2} < 0,$$
and
$$\eta \leq -\frac{g'(1)}{4},$$
and choose $\varphi (s)$ a nonincreasing smooth function which is identically equal to 1 when $s \leq -1$ and identically equal to 0 when $s \geq 0$. 

Then, we introduce $V_*$ the travelling wave solution with minimal speed $c^*$ of the homogeneous problem~\eqref{problemmonohom}, shifted so that $V_* (0) = 1 - \frac{\eta}{2}$. We recall from~\cite{AW} that $V_*$ is a decreasing function and that $V_* ' (s) < 0$ for all $s \in \R$. In particular,
$$\forall s \leq 0 , \quad V_* (s) \geq 1 - \frac{ \eta}{2},$$
and
$$\min_{-1 \leq s \leq 0 } - V_* ' (s) >0.$$
Now we define
\begin{equation}\label{eq:relec0}
\overline{u} (t,x) = V_* (x - c^* t -  M (1 - e^{-\eta t} )) + \eta e^{-\eta t} \varphi (x - c^* t - M (1-e^{-\eta t})),
\end{equation}
where $M$ is a large enough constant so that
\begin{equation}\label{choose_M}
M \min_{-1 \leq s \leq 0 } - V_* ' (s) \geq  \| \varphi '' \|_{L^\infty}.
\end{equation}
Clearly $\overline{u} (t,x)$ is a supersolution of~\eqref{problemmonohom} for all $t \geq 0$ and $x \geq c^* t+M(1-e^{-\eta t})$. On the other hand, for any $t \geq 0$ and $x \leq c^* t+M(1-e^{-\eta t})$:
\begin{eqnarray*}
&&\partial_t \overline{u} - \partial_{xx} \overline{u} - g (\overline{u}) \\
& \geq & - \eta M e^{-\eta t} V_*  ' + \eta e^{-\eta t} \left( -c^* \varphi ' - M \eta e^{-\eta t} \varphi '  -  \varphi '' - \eta \varphi - \frac{g'(1)}{2}  \varphi \right) \\
& \geq & \eta e^{-\eta t} \left( - M V_* ' -  \varphi '' - \frac{g'(1)}{4}  \varphi \right) .
\end{eqnarray*}
If moreover $x \leq c^* t +M(1-e^{-\eta t})- 1$, it immediately follows from the fact that $g'(1) < 0$ and $V_* ' (z) < 0$ for all $z \in \R$ that
$$\partial_t \overline{u} - \partial_{xx} \overline{u} - g (\overline{u}) \geq 0,$$
while if $c^* t +M(1-e^{-\eta t})- 1 \leq x \leq c^* t+M(1-e^{-\eta t})$, the same inequality follows from~\eqref{choose_M}. We conclude that $\overline{u} (t,x)$ is a supersolution of \eqref{problemmonohom}, as well as $\overline{u} (t,x-X)$ for all $X >0$. Thus the function $\overline{u} (t,z + ct - X)$ is a supersolution of \eqref{problemmonobis} for any $X >0$.

Finally, choose $T$ large enough so that $\delta (T)$ defined in \eqref{eq:delta} is such that $\delta(T)\leq \frac{\eta}{4}$, and hence
$$p_+^c (z) + \delta (T) \leq 1 + \frac{\eta}{4} \leq \overline{u} (0,z-X)$$
for all $z \leq X-1$. Notice that the choice of~$T$ does not depend on~$X>0$. Since it is known by phase plane analysis (see either~\cite{AW} or the similar proof of Proposition~\ref{prop:exp_decay} above) that $V_* (s)$ decays at most with the exponential rate $\frac{c^* +\sqrt{c^* \,^2 - 4 g' (0)}}{2}$ as $s \to +\infty$, one can then find~$X$ large enough such that for all $z \in \R$,
$$\min \{ p_+^c (z) + \delta (T), \overline{u}_2 (T,z) \} \leq \overline{u} (0,z  -X).$$
Applying again the comparison principle, we get that for all $t \geq T$ and $z \in \R$,
\begin{equation}\label{nospread_almost}
u(t,z) \leq \overline{u} (t-T,z + c (t-T) -X).
\end{equation}
From the construction of $\overline{u}$ above in~\eqref{eq:relec0}, it follows that spreading does not occur when $c \geq c^*$, even along any time sequence. The second part of Theorem~\ref{th:nospread} is proved.

Now consider the case $c > c^*$. Choose any $\varepsilon >0$ and~$Z>0$ such that $p_+^c (z) \leq \varepsilon$ for all $z \leq -Z$. Putting the inequality~\eqref{nospread_almost} together with Lemma~\ref{lem:sup_cv}, it is straightforward that
$$\limsup_{t\to +\infty , z \in \R} u (t,z) \leq \varepsilon .$$
Letting $\varepsilon \to 0$, this concludes the proof of Theorem~\ref{th:nospread}.
\end{proof}

\subsection{In the case $c\in [2\sqrt{g'(0)},c^*)$}

In this speed range we want to prove that the asymptotic behaviour of $u$ depends on the initial condition (statement $(iii)$ of Theorem \ref{th:regime}), in the sense that both spreading and extinction may occur. More precisely, we will prove the following theorem:
\begin{Th}\label{thm:interm_0} Assume that $c \in [ 2\sqrt{g'(0)},c^*)$.
\begin{enumerate}[$(i)$]
\item There exists some (large enough) nonnegative, bounded and compactly supported initial datum $u_{0,1}$ such that the associated solution satisfies
$$\lim_{t \to +\infty} u_1 (t,z) = p_+^c ,$$
where the convergence is locally uniform in $z$. Moreover, $u_{0,1}$ can also be chosen so that 
$$\min_{z \in \R} p_+^c (z) - u_{0,1} (z) > 0.$$
\item There also exists some (small enough) non trivial and compactly supported initial datum $0 \leq u_{0,2} < 1$ such that the associated solution satisfies
$$\lim_{t \to +\infty} u_2 (t,z) = 0,$$
where the convergence is uniform in $z$.
\end{enumerate}
\end{Th}
%
%

\subsubsection{For large initial conditions}\label{subsec:large_data}

Assume that $c \in [2 \sqrt{g'(0)},c^*)$ and let us find some initial datum $u_{0,1}$ such that the associated solution converges to~$p_+^c$. Recalling our argument in Section~\ref{sec:stationarysol}, and more precisely the third step in the proof of Theorem~\ref{th:stationary}, the extremal trajectory in the phase plane of the homogeneous ODE~\eqref{ODE2} (which converges to $(0,0)$ from the bottom right and lies below any other trajectory converging to $(0,0)$ without crossing the vertical axis) crosses the horizontal axis at some point $(\theta_c ,0)$ with $0 < \theta_c < 1$. 

We now choose $\theta_1 \in (\theta_c,1)$ and denote by $\phi (z)$ the unique solution of \eqref{ODE2} (which we consider here on the whole real line), such that $\phi (0)= \theta_1$ and $\phi ' (0) = 0$. By our choice of $\theta_1$, there must exist some $z_1 >0$ such that $\phi (z_1) = 0$ and $\phi (z)>0 > \phi ' (z)$ for all $z\in [0,z_1)$. Moreover, it is also straightforward by standard phase plane analysis that there exists some $z_2 <0$ such that $\phi (z_2) = 0$, $\phi (z) >0 $ and $\phi ' (z) >0$ for all $z \in (z_2,0]$. Let us now restrict $\phi$ to the interval $[z_2, z_1]$, then extend it to 0 outside this interval. With some slight abuse of notations, we still denote by $\phi$ the resulting function, and note that the function
$$\tilde{\phi}(z)=\phi(z+z_2-1) $$ is a compactly supported subsolution of \eqref{problemmonobis}.

Then define $\underline{u}$ the increasing in time solution of \eqref{problemmonobis} with initial datum $\tilde{\phi}$. Applying Lemma~\ref{lem:sup_cv} and standard estimates, we conclude that $\underline{u}$ converges as $t \to +\infty$ to a stationary solution $p$ of \eqref{problemmonobis} satisfying $\tilde{\phi} \leq p \leq p_+^c$. We claim that $p$ may only be $p_+^c$.
If not, then $p \equiv p_\alpha$ for some $\alpha \in (0,\alpha^*_c]$ where $(p_\alpha)_\alpha$ denotes the family of ground states constructed in Section~\ref{sec:stationarysol}. It again follows from the third step in the proof of Theorem~\ref{th:stationary} that, for any $\alpha \in (0,\alpha^*_c]$, the trajectory in the phase plane of \eqref{ODE2} associated with $p_\alpha$ lies `above' the extremal trajectory and, in particular, it crosses the horizontal axis at some point $(\theta_\alpha , 0)$ with $0 < \theta_\alpha \leq \theta_c$. Therefore, we get that for some $\alpha \in (0,\alpha^*_c]$,
$$\max_{z \in \R} p  = \theta_\alpha \leq \theta_c .$$
However, this clearly contradicts our choice of $\theta_1$ above and the fact that, for all $t >0$,
$$\max_{z \in \R} p (z) \geq \max_{z \in \R} \underline{u} (t,z) \geq \max_{z \in \R} \tilde{\phi} (z) = \theta_1 .$$
We conclude that $\underline{u} (t,z)$ converges as $t \to +\infty$ to $p_+^c$ locally uniformly with respect to~$z$, and since $\underline{u}$ is increasing in time and $\underline{u} (0,\cdot) \equiv \tilde{\phi} (\cdot)$ has compact support, it becomes straightforward that $\min_\R p_+^c - \tilde{\phi} >0$. Finally, it follows by the comparison principle that spreading occurs for any bounded initial datum $u_{0,1} \geq \tilde{\phi}$.

\subsubsection{For small initial conditions}

We again assume that $c \in [2 \sqrt{g'(0)},c^*)$ and we prove that for some small initial data, the solution of \eqref{problemmonobis} converges to~0 uniformly with respect to $z$ as time goes to infinity.  We introduce
$$\hat{\phi} = \min \{ p_\alpha , p_{\alpha^*_c} \},$$
where $\alpha \in (0, \alpha^*_c)$. Then $\hat{\phi}$ is a (generalized) supersolution of \eqref{problemmonobis}.

Looking back at our construction of ground states, it is clear from the proof of Theorem~\ref{th:stationary} (recall~\eqref{eq:left}) that 
$$\forall z < 0 , \quad \hat{\phi} (z) = p_\alpha (z),$$
and that, by Proposition~\ref{prop:exp_decay}, there exists $z_0 >0$ such that
$$\forall z > z_0 , \quad \hat{\phi} (z) =p_{\alpha^*_c} (z).$$
Therefore, the solution $\overline{u}$ of \eqref{problemmonobis} with initial datum $\hat{\phi}$ is decreasing in time, and hence it converges locally uniformly to a bounded and nonnegative stationary solution~$p$ such that
$$p \leq \min \{ p_\alpha , p_{\alpha^*_c} \}.$$
Clearly $p$ is either a ground state or the trivial solution~0. However, for any $\alpha \in (0,\alpha^*_c)$ and proceeding as above, we have that $p_{\alpha^*_c} (z) - p_\alpha (z)$ is strictly positive for negative $z$, and strictly negative for large positive~$z$. Thus, $p \equiv 0$.

Let us now briefly check that the convergence is in fact uniform with respect to $z \in \R$. Choose any small $\varepsilon >0$, and let $Z$ such that
$$\sup_{|z| \geq Z } \hat{\phi} (z) \leq \varepsilon.$$
Then, by monotonicity with respect to time, it follows that $\overline{u} (t,z) \leq \varepsilon$ for all $t>0$ and $|z| \geq Z$. Together with the locally uniform convergence to $p \equiv 0$ proved above, it is straightforward to conclude that
$$\lim_{t\to +\infty} \sup_{z \in \R} \overline{u} (t,z) =0.$$
It of course follows that, for any initial datum $0 \leq u_{0,2} \leq \hat{\phi}$, the associated solution converges uniformly to~0. This ends the proof of Theorem~\ref{thm:interm_0}.

\section{Exponential bound and convergence in the moving frame}\label{sec:convstatsol}

In this section we will prove Proposition~\ref{prop:conv} which states that the solution $u(t,z)$ of \eqref{problemmonobis} converges as time goes to infinity, for all $c \geq 0$, to a stationary solution that is either 0, the critical ground state $p_{\alpha^*_c}$ (when it exists) or the invasion state~$p_+^c$. We already know from the previous section that for $c<2\sqrt{g'(0)}$, $u(t,z)$ converges locally uniformly to $p_+^c$ as time goes to infinity, whereas when $c>c^*$, $u(t,z)$ converges uniformly to 0 as time goes to infinity. 

Therefore, we only consider throughout this section the remaining case $c \in [ 2 \sqrt{g' (0)}, c^*]$ (note that we allow this interval to be reduced to a singleton as in the KPP case). While the convergence to a stationary solution follows a zero number argument inspired from Du and Matano~\cite{DM}, the fact that grounding may only occur to the critical ground state (that is, the solution may never converge to a ground state other than~$p_{\alpha^*_c}$) relies on uniform in time exponential bounds on the solution $u(t,z)$ as $z \to +\infty$, see Proposition~\ref{th:boundexp} below.

\subsection{Exponential bound of the solution}\label{sec:energy}

We have already shown in Section~\ref{sec:expdecay} that the critical ground state $p_{\alpha^*_c} (z)$ distinguishes itself from the other ground states by its faster decay as $z \to +\infty$. Therefore, we now take interest in the behaviour of $u(t,z)$ the solution of \eqref{problemmonobis} as $z \to +\infty$. It turns out that, for compactly supported initial data and unless spreading occurs, the solution $u(t,z)$ satisfies uniform in time exponential bounds which will immediately rule out non critical ground states in the large-time asymptotics. More precisely, we now prove the following:
\begin{Prop}\label{th:boundexp}
 Let $c \in [2 \sqrt{g'(0)},c^*]$. Let also $u_0$ be a bounded, nonnegative, non trivial and compactly supported function, and~$u$ be the associated solution of~\eqref{problemmonobis}. Assume that $u$ does not spread (in the sense of Definition~\ref{def:main}).
\begin{itemize} 
\item If $c\in(2\sqrt{g'(0)},c^*]$, then there exist $\epsilon>0$ and $Z_\epsilon>0$ such that for any $t\in(0,+\infty)$ and $z\geq Z_\varepsilon$,
$$u(t,z)\leq \epsilon e^{-\lambda_\gamma (c)(z- Z_\varepsilon)},$$
where
$$\lambda_\gamma (c) :=\frac{c+\sqrt{c^2-4(g'(0)+\gamma)}}{2},$$
for some $\gamma>0$ such that $c>2\sqrt{g'(0)+\gamma}$.
\item If $c=2\sqrt{g'(0)}$, then there exists $Z_r>0$ such that for all $t\in(0,+\infty)$ and $z\geq Z_r$,
$$u(t,z)\leq (1+\sqrt{z - Z_r})e^{-\frac{c}{2} (z- Z_r) }.$$
\end{itemize}
\end{Prop}
\begin{Rk}
In the second part of Proposition~\ref{th:boundexp}, we use the notation $Z_r$. The reason is that when $c = 2 \sqrt{g'(0)}$, we need the $C^{1,r}$-regularity of $g$ to get the exponential bound, which is not the case when $c > 2 \sqrt{g'(0)}$.\medskip
\end{Rk}

Notice that the above Proposition~\ref{th:boundexp} is closely related to the energy approach introduced in~\cite{Heinze} and used in several papers to prove convergence to travelling wave solutions~\cite{GR,M04,MN12,MZ,Risler}. We also refer to~\cite{BN} where the energy approach was used in the context of climate change models. Indeed, consider the functional
$$E_c[w] :=\int_\R e^{cz}\left\{\frac{(w')^2}{2}-F(z,w)\right\},$$
with $F (z,w) = \int_0^s f(z,s) ds$, and which is well-defined for any function~$w$ in the weighted Sobolev space $H^1 (\R , e^{cz} dz)$. The functional $E_c $ is the natural energy associated with \eqref{problemmonobis} in the sense that \eqref{problemmonobis} is a gradient flow generated by~$E_c$.

Then, when $c > 2 \sqrt{g'(0)}$, Proposition~\ref{th:boundexp} implies that the energy $E_c [ u(t,\cdot)]$ of the solution remains bounded uniformly in time. Using the energy functional as a Lyapunov function, one can then prove the large-time convergence to some stationary state. Such an approach has been used especially in~\cite{MZ} where sharp dichotomy results were also obtained by observing that the energy remains bounded as time goes to infinity if and only if spreading does not occur.

Note however that in the critical case $c = 2\sqrt{g'(0)}$, Proposition~\ref{th:boundexp} does not guarantee the boundedness of the energy in infinite time, which is why we chose a different approach similar to~\cite{DM}.

\begin{proof}[Proof of Proposition~\ref{th:boundexp}]
Throughout this proof, we assume that $c \in [2 \sqrt{g'(0)},c^*]$ and that spreading does not occur, i.e. the initial datum $u_0$ is chosen so that $u(t,z)$ does not converge (locally uniformly with respect to $z$) to $p_+^c(z)$ as $t\to+\infty$. 

Let us first claim that there also does not exist any sequence $(t_n)_{n \in \mathbb{N}}$ such that $u(t_n,z)$ converges locally uniformly to $p_+^c (z)$ as $n \to +\infty$. If $c= c^*$, this was already proved in Theorem~\ref{th:nospread}. When $c < c^*$, we proceed by contradiction and assume that such a time sequence exists. In particular, there exists $n$ large enough so that, for all $z \in \R$,
$$u(t_n,z) \geq u_{0,1} (z).$$
Here $u_{0,1} < p_+^c$ is a compactly supported function given by Theorem~\ref{thm:interm_0} and is such that the associated solution $u_1$ of \eqref{problemmonobis} converges to $p_+^c$ locally uniformly in $z$ as time goes to infinity. Then, by the comparison principle
$$u(t_n+ t,z) \geq u_1 (t,z),$$
for all $t>0$ and $z \in \R$, which passing to the limit as $t \to +\infty$ contradicts our assumption that spreading in the sense of Definition~\ref{def:main} does not occur. We conclude as announced that $u$ does not converge to $p_+^c$ even along any time sequence.

We will now show that $u$ admits an exponential bound as in Proposition \ref{th:boundexp}. The proof is a succession of several lemmas which follow.
\begin{Lemma}\label{lem:energy1}
Under the same assumptions as in Proposition~\ref{th:boundexp}, let $u$ be the solution of \eqref{problemmonobis} with initial datum $u_0$.

Then there exist $\delta>0$ and $A_\delta >0$ such that 
\begin{equation}\label{energy1_2}
\underset{t \geq 0 \: z \geq A_\delta}{\sup} \: u(t,z) \leq 1 -\delta.
\end{equation}
\end{Lemma}
\begin{proof}
We begin by showing that~\eqref{energy1_2} holds true if there exists $T_\delta > 0 $ such that
\begin{equation}\label{energy1_1}
\underset{t\geq T_\delta,\:z \in\R}{\sup}\: u(t,z)\leq 1-\delta.
\end{equation}
Indeed, assume that \eqref{energy1_1} holds for some $\delta >0$ and~$T_\delta >0$. As we have already mentioned in Section~\ref{sec:large_c}, we have for all $t \geq 0$ and $z \in \R$ that
$$u(t,z) \leq \frac{e^{At}}{\sqrt{4 \pi t}} \int_\R u_0 (y) e^{-\frac{|z+ct-y|^2}{4t}} dy,$$
where $A = \sup_{0\leq u \leq M} \frac{g(u)}{u}$ and $M > \max \{ \|u_0 \|_{L^\infty (\R )} , 1 \}$. Since $u_0$ has a compact support, one can find $A_\delta >0$ such that, for all $z \geq A_\delta$ and $0 \leq t \leq T_\delta$,
$$u (t,z) \leq \frac{e^{At}}{\sqrt{4 \pi t}} \int_\R u_0 (y) e^{-\frac{|z+ct-y|^2}{4t}} dy \leq 1 -\delta .$$
Together with~\eqref{energy1_1}, this implies that~\eqref{energy1_2} holds true.

Let us now prove \eqref{energy1_1}. We already know by Lemma~\ref{lem:sup_cv} that
$$  \lim_{t \to +\infty} \ \sup_{z \in \R} u (t,z) - p_+^c (z) = 0,$$
hence
$$  \limsup_{t \to +\infty} \ \sup_{z \in \R} u (t,z) \leq 1 .$$
Therefore, we can argue by contradiction by assuming that there exist some sequences $t_n \to +\infty$ and $z_n \in \R^\mathbb{N}$ such that $u(t_n,z_n)\to1$ as $n\to+\infty$. If $(z_n)_n$ is bounded in $\R$ then up to a subsequence, $z_n \to z_\infty \in \R$ and the sequence $u_n (t,z) := u(t_n+t,z)$ converges locally uniformly to $u_\infty$ a solution of \eqref{problemmonobis} such that $u_\infty (0,z_\infty) = 1$ and $u_\infty (t,z) \leq 1$ for all $t \in \R$ and $z \in \R$.  As $1$ is a strict supersolution of \eqref{problemmonobis} and applying the strong maximum principle, we reach a contradiction. If $(z_n)_n$ is unbounded from below, then a similar argument leads to the same contradiction. 

Lastly, we assume that $z_n \to +\infty$. Notice first that when $c=c^*$ the inequality \eqref{nospread_almost} in the proof of Theorem \ref{th:nospread} implies that there exists some $X>0$, $T>0$ such that for all $t>T$, $z\in \R$,
$$u(t,z)\leq V_*(z-X-M(1-e^{-\eta (t-T)}))+\eta e^{-\eta(t-T)}\varphi(z-X - M(1-e^{-\eta (t-T)})),$$
and $V_*$, $\varphi$ are two functions that converge to 0 at~$+\infty$. Thus the case $z_n\to+\infty$ cannot happen when $c=c^*$, so that $c<c^*$. Define $u_n (t,z) := u (t_n+ t,z_n+ z)$, which up to the extraction of some subsequence converges to $u_\infty$ satisfying the equation
$$\partial_t u_\infty  - \partial_{zz} u_\infty - c \partial_z u_\infty = g (u_\infty).$$
Furthermore, $u_\infty (0,0) = 1$ and $u_\infty (t,z) \leq 1$ for all $t \in \R$ and $z \in \R$. Applying again the strong maximum, we infer that $u_\infty \equiv 1$. Then there exists $n$ large enough such that, for all $z \in \R$,
$$u (t_n,  z) \geq u_{0,1} (z - z_n),$$
where $u_{0,1}$ is given by Theorem~\ref{thm:interm_0}. Note also that, from the proof of Theorem~\ref{thm:interm_0} and without loss of generality, we may assume that $u_{0,1} (z)$ is a subsolution of \eqref{problemmonobis}, and so is $u_{0,1} (z - z_n)$ thanks to the fact that $z_n >0$, for $n$ large enough. Thus, the solution~$\tilde{u}_1 (t,z)$ of \eqref{problemmonobis} with initial datum $u_{0,1} (z-z_n)$ is bounded and increasing in time. In particular, it converges by standard parabolic estimates to a stationary solution~$p_n$. 

On the other hand, we also know from Theorem~\ref{thm:interm_0} that $u_1 (t,z)$ the solution of \eqref{problemmonobis} associated with $u_{0,1} (z)$ converges locally uniformly to $p_+^c (z)$ as $t \to +\infty$. By the comparison principle and since $u_1 (t,z-z_n)$ is also a subsolution of \eqref{problemmonobis}, we get that
$$u(t_n+ t,z) \geq \tilde{u}_1 (t,z) \geq u_1 (t,z-z_n),$$
for all $t \geq 0$ and $z \in \R$. Thus $p_n (z) \geq p_+^c (z-z_n)$ for all $z \in \R$, $n$ large enough, and it immediately follows from Theorem~\ref{th:stationary} that $p_n (z) = p_+^c (z)$ for all $z \in \R$. We conclude that, for all $Z >0$,
$$\liminf_{t \to +\infty} \inf_{|z|\leq Z} \left( u(t,z) - p_+^c (z) \right) \geq 0.$$
In other words (recall also Lemma~\ref{lem:sup_cv}), we have just proved that spreading occurs, which again contradicts our hypotheses. This proves~\eqref{energy1_1}, hence~\eqref{energy1_2}.

\end{proof}

\begin{Lemma}\label{lem:phizerodelta}
Consider $0 < \delta <1$ and the family of phase plane solutions of the homogeneous monostable equation~\eqref{ODE2}, namely
$$\phi '' + c \phi ' + g (\phi) = 0,$$
such that $\phi'(0)=0$ and $\phi(0)=\theta \in(0,1-\delta ]$. Then there exists $z_\theta<0$ such that $\phi(z_\theta)=0 < \phi ' (z_\theta)$ and $\phi (z) >0$ for all $z \in (z_\theta,0]$. Furthermore,
$$\underset{\theta\in(0,1-\delta ]}{\sup}\: |z_\theta|\leq K_\delta<+\infty.$$
\end{Lemma}
\begin{proof}
The existence of $z_\theta$ is immediate from looking at the phase plane, so we only need to prove that the mapping $\theta \mapsto z_\theta$ is uniformly bounded in $ (0,1-\delta)$. We will do this by looking at the ODE on the interval~$[z_\theta,0]$. Note that on this interval, the function $\phi$ is increasing and thus lies in the range $[0,1-\delta]$. Then for $z \in (z_\theta, 0]$:
$$0 = \phi '' + c \phi ' + g( \phi)>\phi '' + \gamma_\delta \phi ,$$
where $\gamma_\delta >0$ is such that $g(s) > \gamma_\delta s$ for all $0 \leq s \leq 1 -\delta$. 
It follows that in the phase plane, the trajectory of $\phi$ lies strictly above (at least while it remains in the $\{p,p'>0\}$ part of the phase plane, which is the case if we consider the restriction of $\phi$ to the interval $[z_\theta,0]$) that of $\psi$ the solution of 
$$\psi''+\gamma_\delta \psi=0,$$
such that $\psi(0)=\phi(0)=\theta$ and $\psi'(0)=\phi'(0)=0$. In particular, it is straightforward that $\phi < \psi$ on $[z_\theta,0)$. Indeed, since $\phi '' (0) < \psi '' (0)$, it is clear that $\phi < \psi$ on some interval $(0-\eta,0)$ where $\eta >0$. Let us proceed by contradiction and assume that there exists $a \in [z_\theta,0)$ such that $\phi (a) = \psi (a)$, and $\phi(z) < \psi (z)$ for all $z \in (a,0)$. Then, as the trajectory of $\phi$ lies above that of $\psi$, we get that $\phi ' (a) > \psi '(a)$, which contradicts our choice of $a$. We conclude as announced that $\phi < \psi$ on $[z_\theta,0)$.

Now, clearly $\psi(z)=\theta\cos(\sqrt{\gamma_\delta}z)$ and there exists $K_\delta :=  \frac{\pi}{2 \sqrt{\gamma_\delta}}>0$ which does not depend on $\theta \in (0,1-\delta)$ such that $\psi(-K_\delta)=0$. Therefore, there exists $z_\theta \geq- K_\delta$ such that $\phi (z_\theta) = 0$ and $\phi (z) >0$ for all $z \in (z_\theta,0]$. In particular $\phi ' (z_\theta)\geq 0$ and, in fact, $\phi ' (z_\theta) > 0$ since $\phi \not \equiv 0$, which concludes the proof. Note though that $K_\delta\to+\infty$ as $\delta\to 0$, which of course is unavoidable since $1$ is solution of the homogeneous monostable equation.
\end{proof}

\begin{Lemma}\label{lem:energy3}
Under the same assumptions as in Proposition~\ref{th:boundexp}, let $u$ be the solution of \eqref{problemmonobis} with initial datum $u_0$, and denote $K >0 $ such that $support (u_0) \subset (-K,K)$.

Then for all $t\in(0,+\infty)$, $z\mapsto u(t,z)$ is nonincreasing in $[D_0,+\infty)$, where $D_0 :=K+K_\delta + A_\delta$ with $\delta$, $A_\delta$ given by Lemma~\ref{lem:energy1} and $K_\delta$ given by Lemma~\ref{lem:phizerodelta}.
\end{Lemma}
\begin{proof} We only consider the case when $u_0$ is of class $C^1$ and there exists $0<a< K$ such that $\partial_z u_0(a) < 0$ and $\partial_z u_0(z) \leq 0$ for $z \geq a$. Indeed, by continuity of the solution of \eqref{problemmonobis} with respect to the initial datum and noting that any compactly supported function $u_0$ can be approached by a sequence of functions satisfying these additional hypotheses, the lemma eventually follows in the general case.

Since $a >0$, one can use standard estimates and find some $\tau >0$ small enough such that $\partial_z u(t,a)<0$ for all $t \in (0,\tau)$. As $\partial_z u$ satisfies 
$$\partial_t (\partial_z u ) (t,z) = \partial_{zz} (\partial_z u ) (t,z) + c \partial_z (\partial_z u ) (t,z) + \partial_z u (t,z) g' (u (t,z)),$$
for all $t >0$  and $z \geq a$, one can apply the strong maximum principle to get that 
$$\forall t \in (0,\tau), \  \forall z \geq a , \quad \partial_z u<0 .$$
Now argue by contradiction and assume that there exist $t_1>0$, $z_1\geq D_0$ such that $\partial_z u(t_1,z_1)=0$. Let $\phi$ be the solution of the homogeneous ODE~\eqref{ODE2} with 
$$\phi(0)=u(t_1,z_1), \: \phi'(0)=0 = \partial_z u (t_1,z_1).$$
Letting $\theta := u(t_1,z_1) $, which by Lemma~\ref{lem:energy1} and our choice of $D_0 > A_\delta$ satisfies $\theta \leq 1 -\delta$, we know from Lemma~\ref{lem:phizerodelta} that there exists $z_\theta \in [-K_\delta,0)$ such that $\phi(z_\theta)=0$, $\phi ' (z_\theta) >0$ and $\phi >0$ in $(z_\theta,0]$. Moreover, it also follows from our choice of $D_0$ that $z_1+z_\theta \geq K + A_\delta > K$. On the other hand, $\phi$ may either be positive on $[0,+\infty)$ (and thus on $(z_\theta,+\infty)$), or admits some smallest $z_2 >0$ such that $\phi(z_2)=0$ and $\phi ' (z_2) <0$. 

While we treat those two cases separately below, both parts rely on a so-called `zero number' argument that we detail in Appendix \ref{A:number_zero}. The main idea is that, denoting by $h\not\equiv0$ the solution of a one-dimensional linear parabolic equation of the type
$$\partial_t h (t,z) = \partial_{zz} h (t,z)  + c h (t,z) + b(t,z) h (t,z), \quad t \in (t_1,t_2), \ z \in I ,$$
where $I$ is an interval, $c$ is a constant and $b$ a bounded function, then the number of sign changes $\mathcal{Z}_I [ h (t,\cdot) ]$ of $h (t,\cdot)$ on $I$ is a nonincreasing function of $t \in (t_1,t_2)$ and the zeros do not accumulate in $I$. We refer to Proposition~\ref{prop:chang_signR} in the appendix and to \cite{A,DM,matano} for more details.
\begin{itemize}
\item We first consider the case when $\phi (z)$ is positive in $(z_\theta,+\infty)$. Then, by standard phase plane analysis, it converges exponentially to 0 as $z \to +\infty$, and the exponential convergence rate is given by any of the two positive eigenvalues 
$$\lambda_{\pm} (c) = \frac{c \pm \sqrt{c^2 - 4g'(0)}}{2}$$
of the linearization of~\eqref{ODE2} around 0. As $u_0$ has a compact support included in the interval $(-K,K)$, and recalling that $z_1 + z_\theta \geq K+ A_\delta$, we know that
$$\mathcal{Z}_I [u_0(\cdot)-\phi(\cdot-z_1)]=1,$$
where $I=[z_1+z_\theta - \eta ,+\infty)$ with $\eta  \in (0,A_\delta)$ such that $\phi (z) < 0$ for all $z \in (z_\theta - \eta, z_\theta )$. Note that such a $\eta$ clearly exists since $\phi ' (z_\theta) >0$. Putting together the facts that $u_0 \equiv 0$ in $(K,+\infty)$, that $\partial_z u<0$ for all $t\in(0,\tau)$ and $z\geq z_1 + z_\theta - \eta \geq a$, and that one can show as in Section~\ref{sec:large_c} that, for all $t>0$ and $z \in \R$,
\begin{equation}\label{correct_1}
u (t,z) \leq \frac{e^{At}}{\sqrt{4\pi t}} \int_\R u_0 (y) e^{- \frac{|z+ct-y|^2}{4t}} dy \quad \mbox{ where } A>0,
\end{equation}
then one can check up to reducing $\tau >0$ that
$$\mathcal{Z}_I [u (\tau,\cdot)-\phi(\cdot-z_1)]=1.$$
Moreover, $h (t,z) := u(t,z) - \phi (z-z_1)$ satisfies the equation
\begin{equation}\label{eq:w1}	
\partial_t h (t,z) = \partial_{zz} h (t,z)  + c \partial_z h (t,z) + b(t,z) h (t,z),
\end{equation}
for all $t >0$ and $z \in I$, where 
$$b (t,z):= \left\{
\begin{array}{l}
\displaystyle \frac{g(u (t,z)) - g (\phi (z -z_1))}{u(t,z) - \phi (z-z_1)} \mbox{ if } u(t,z) \neq \phi (z-z_1),\vspace{3pt}\\
g' (u(t,z)) \mbox{ if } u(t,z) = \phi (z-z_1).
\end{array}\right.$$

Note also that, since $u(t,z) >0$ for all $t >0$ and $z \in \R$ by the strong maximum principle, it is clear that $h (t,z_1 + z_\theta -\eta) = u (t,z_1+z_\theta - \eta ) - \phi (z_\theta - \eta) >0$ for all $t \geq0$. Therefore, from Proposition~\ref{prop:chang_signR}, we infer that
$$t \in (0,+\infty) \mapsto \mathcal{Z}_I[h (t,\cdot)]$$
is a nonincreasing function and furthermore, whenever $h (t^*,\cdot)$ admits a degenerate zero in the interior of $I$ for some $t^* >0$, then 
$$\mathcal{Z}_I [h (s_1,\cdot)] > \mathcal{Z}_I [h (s_2,\cdot)]$$
for any $s_1 \in (0,t^*)$ and $s_2 > t^*$. Here we know that such a degenerate zero occurs at time~$t_1$ and point~$z_1$, and it follows that, for all $t>t_1$,
$$\mathcal{Z}_I [h (t,\cdot)] = 0 .$$
Using again the fact that $h (t,z_1 + z_\theta - \eta) >0$ for all $t>0$, we get more precisely that,
$$\forall t > t_1 , \ \forall z \geq z_1 + z_\theta - \eta, \quad u(t,z) \geq \phi (z - z_1).$$
However, from~\eqref{correct_1} and recalling that $\phi (z)$ decays exponentially as $z \to +\infty$, we have finally reached a contradiction.
\item Now consider the case when $\phi$ changes sign in $(0,\infty)$. More precisely, there are $z_\theta < 0 < z_2$ such that $\phi (z_\theta) = \phi (z_2) =0$, and $\phi (z) >0 $ for all $z \in (z_\theta,z_2)$. Moreover, as we already mentioned above, $ \phi' ( z_\theta) > 0 > \phi ' (z_2)$, hence there also exists some $\eta \in (0,A_\delta)$ such that $\phi (z) <0$ for all $z \in (z_\theta - \eta, z_\theta) \cup (z_2 , z_2 +\eta)$.

Proceeding as before, we let $h(t,z) := u (t,z) - \phi (z - z_1)$ which again solves~\eqref{eq:w1} for all $t>0$ and $z \in I := [z_1 + z_\theta - \eta, z_1 + z_2 + \eta]$, where~$b$ is defined in a similar fashion. Clearly, since $z_1 + z_\theta - \eta \geq K$, we have
$$\mathcal{Z}_I [h (0,\cdot)] = 2,$$
as well as $h(t,z_1 + z_\theta - \eta) >0$ and $h(t,z_1+ z_2 + \eta) >0$ for all $t>0$. Using again the behaviour of $u$ when $t$ is close to 0, we get as before that, up to reducing $\tau$, 
$$\mathcal{Z}_I [h (\tau,\cdot)] = 2.$$
Therefore, applying again Proposition~\ref{prop:chang_signR}, we get that the function $t \mapsto \mathcal{Z}_I [h (t,\cdot)]$ is nonincreasing on $(0,\infty)$. Furthermore, because $h (t_1, \cdot)$ admits a degenerate zero at the point $z_1$, we have that
$$\mathcal{Z}_I [h (t,\cdot)] \leq 1$$
for all $t  > t_1$. In fact, since $h (t,z_1+ z_\theta - \eta)$ and $h (t,z_1+z_2+\eta)$ are both positive for all $t >0$, we even get that
$$\forall t > t_1, \ \forall z \in I ,\quad u(t,z) \geq \phi (z -z_1).$$
Unlike in the previous case, this is not enough to reach a contradiction. Applying the comparison principle, we infer that $u(t_1 + t,z) \geq \underline{u} (t,z)$, where $\underline{u}$ denotes the solution of \eqref{problemmonobis} with initial datum
$$\tilde{\phi}(z) := 
\left\{
\begin{array}{l}
\phi (z-z_1) \mbox{ if } z \in [z_1 +z_\theta, z_1 + z_2],\vspace{3pt}\\
0 \mbox{ otherwise.}
\end{array}
\right.$$
Proceeding as in Section~\ref{subsec:large_data}, one can check that $\underline{u} (t,z)$ converges to $p_+^c (z)$ as $t \to +\infty$, locally uniformly with respect to $z \in \R$. This contradicts our assumption that the solution $u$ does not spread.
\end{itemize}
In both cases we have reached a contradiction. Lemma~\ref{lem:energy3} is proved.
\end{proof}

\begin{Lemma}\label{uepsilonthm}
Under the same assumptions as in Proposition \ref{th:boundexp}, let $u$ be the solution of \eqref{problemmonobis} with initial datum $u_0$. Then, for all $\epsilon>0$, there exists $Z_\epsilon \in \R$ such that 
$$\forall t \geq 0 , \ \forall z \geq Z_\varepsilon, \quad u(t,z)\leq\epsilon .$$
\end{Lemma}
\begin{proof}
When $c=c^*$, this comes from the proof of Theorem~\ref{th:nospread} in Section~\ref{sec:large_c} (see in particular the inequality~\eqref{nospread_almost}, where the right-hand side decays as $z \to +\infty$, uniformly with respect to $t>0$). When $c<c^*$ we argue by contradiction and assume that there exists $\epsilon>0$, such that for any $Z \in \R$, there exists $t_Z>0$ and $Z' >Z$ with
$$u(t_Z,Z')\geq \epsilon.$$
Choose $Z >  D_0$ where $D_0$ is given by Lemma~\ref{lem:energy3}, and then
$$\underset{D_0\leq z\leq Z'}{\inf}\:u(t_Z,z)\geq \epsilon.$$
We claim that, provided $Z'$ is large enough, this implies that spreading occurs.

First, let $v(t,z)$ be the solution of the homogeneous monostable equation 
$$\partial_t v = \partial_{xx} v + g (v),$$
with initial datum $v_0(x)=\epsilon\chi_{(-D_1,D_1)}(x)$, where $\chi$ denotes the characteristic function and $D_1 >0$. As follows from classical results~\cite{AW}, the function $v(t,x)$ converges as $ t \to +\infty$ to 1, locally uniformly with respect to $x \in \R$. In particular, there exists $T>0$ such that for all $x \in \R$,
$$v(T,x)> u_{0,1} (x),$$
where $u_{0,1} < p_+^c < 1$ is a compactly supported function, given by Theorem~\ref{thm:interm_0} so that $u_1 (t,z)$ the solution of \eqref{problemmonobis} with initial condition $u_{0,1}$ spreads in the sense of Definition~\ref{def:main}.

Now let $\tilde{v}$ be the solution of the homogeneous monostable equation but in a bounded domain with Dirichlet boundary conditions on the boundary, i.e.
\bee\begin{cases}
\partial_t\tilde{v}-\partial_{xx}\tilde{v}=g(\tilde{v}), &t\in(0,+\infty),\:x\in(-D_2,D_2), \vspace{3pt}\\
\tilde{v}(t,\pm D_2)=0,&t\in(0,+\infty), \vspace{3pt}\\
\tilde{v}(0,x)=v_0(x), &x\in(-D_2,D_2),
\end{cases}\eee
where $D_2 > D_1$. Since $\tilde{v}$ converges locally uniformly in time and space to $v$ as $D_2 \to +\infty$, there exists $D_2$ large enough such that $\tilde{v}(T,x)\geq u_{0,1} (x)$. 

Up to increasing $Z$ above, we can assume without loss of generality that
\begin{equation}\label{eqn:choice_Z}
Z ' - D_0 \geq 2 D_2 + c T .
\end{equation}
By construction, for all $z \in \R$,
$$u (t_Z, z) \geq v_0 (z - Z' + D_2),$$
and for all $t \in [0,T]$, we have
$$u(t_Z + t,Z' - 2 D_2 - c t) \geq 0 = \tilde{v} (t, -D_2) \quad \mbox{ and } \quad u (t_Z + t , Z' -ct) \geq 0 = \tilde{v} (t,D_2) .$$
Thanks to~\eqref{eqn:choice_Z}, one can check that $(t,z) \mapsto \tilde{v} (t, z - Z' + D_2 + ct)$ satisfies \eqref{problemmonobis} for all
$$(t,z) \in Q_T := \{ (t,z) \, | \ 0 < t < T \  \mbox{ and } \ Z' -2 D_2 - ct < z < Z' - ct \} .$$
Therefore, we can apply the comparison principle and conclude that, for all $(t,z) \in Q_T$,
$$u(t_Z + t ,z) \geq \tilde{v} (t,z- Z' + D_2 + ct).$$
In particular, for all $z \in \R$,
$$u(t_Z + T,z) \geq u_{0,1} (z - Z' + D_2 + cT).$$
Using again the fact that $ Z' - D_2 - cT >0$, it is straightforward that $u_1 (t,z - Z' + D_2 + cT)$ is a subsolution of \eqref{problemmonobis} (recalling that $u_1(t,z)$ is the solution of \eqref{problemmonobis} with initial datum $u_{0,1}$). Another application of the comparison principle and of Lemma~\ref{lem:sup_cv} leads to the conclusion that $u$ spreads in the sense of Definition~\ref{def:main}, which again contradicts our assumption. We proved the lemma.
\end{proof}
Now we can prove Proposition \ref{th:boundexp}.
Let us first assume that $c \in (2 \sqrt{g'(0)},c^*]$, and choose $\gamma >0$ such that $c > 2 \sqrt{g'(0) +\gamma}$. Let then $\epsilon>0$ be small enough so that $g ' (u)<g'(0)+\gamma$ for all $u\in[0,\epsilon]$. Define 
$$v(z)=\epsilon e^{-\lambda_\gamma (c) (z-Z_\varepsilon)},$$
where $Z_\epsilon >0$ comes from Lemma~\ref{uepsilonthm}. Noting that $\lambda_\gamma (c)$ is the largest root of $\lambda^2 - c \lambda + g'(0) + \gamma$ and from our choice of $\varepsilon$, we get that for all $z\geq Z_\varepsilon$,
\begin{align*}
v_t-v_{zz}-cv_z-g(v)&=-\lambda_\gamma (c)^2v+c \lambda_\gamma (c) v-g(v) \geq 0.
\end{align*}
Moroever, up to increasing $Z_\varepsilon$ and without loss of generality, we have that $u_0 \equiv 0 \leq v$ in $(Z_\varepsilon, + \infty)$. Using Lemma~\ref{uepsilonthm}, we also know that 
$$u(t, Z_\varepsilon )\leq \varepsilon = v(Z_\varepsilon),$$
for all $t \geq 0$. Therefore, we conclude applying the parabolic maximum principle that 
$$\forall z\geq Z_\varepsilon, \ t\geq 0, \quad u(t,z)\leq v(z) ,$$
and the wanted inequality immediately follows.

We now consider the case $c=2\sqrt{g'(0)}$ and define, for all $z\geq0$,
$$v(z)=(1+\sqrt{z})e^{-\frac{c}{2}z},$$
Let also $Z >0$ be such that, for all $z \geq Z$, one has $0 \leq v (z) \leq 1$ and hence
$$| g(v) - g' (0) v  | \leq   \int_0^v \left|g'(s) - g'(0)\right| ds  \leq C_r v^{1+ r},$$
where $C_r >0 $ comes from the $C^{1,r}$-regularity of $g$. Then one can check that for all $z \geq Z$,
\begin{align*}
v_t-v_{zz}-cv_z-g(v)&=g'(0)v-g(v)+\frac{1}{4}\frac{e^{-\frac{c}{2}z}}{z\sqrt{z}}\\
				&\geq -C_r v^{1+r} +\frac{1}{4}\frac{e^{-\frac{c}{2}z}}{z\sqrt{z}},\\
				& \geq e^{-\frac{c}{2} z} \left( - C_r  (1 + \sqrt{z})^{1+r} e^{-\frac{cr}{2} z} + \frac{1}{4 z \sqrt{z}}\right)\\
 				& >0 ,
\end{align*}
and the last inequality holds up to increasing $Z$ (depending on $r$).

On the other hand, since 
$$\min_{z \in [0,Z]} v(z) >0 ,$$
it follows from Lemma~\ref{uepsilonthm} that there exists $Z' >0 $ such that, for all $t\geq 0$ and $z \in [Z' - Z, Z']$,
$$u(t, z ) \leq v(z - Z' +Z) .$$
Up to increasing $Z'$ and since $u_0$ has compact support, we can also assume that $u_0 (z) = 0 \leq v (z - Z' + Z)$ for all $z \in [Z' - Z, +\infty)$.
Therefore, we apply the comparison principle again and conclude that for all $t \geq 0$ and $z \geq Z' - Z $,
$$u(t,z) \leq v (z - Z' + Z) =  \left((1 + \sqrt{z-Z' + Z} \right) e^{-\frac{c}{2} (z - Z' +Z)}.$$
Letting $Z_r = Z' - Z$, we reach the wanted inequality and the proposition is proved.

\end{proof}

\subsection{Convergence in the moving frame}\label{sec:convcritic}

In this subsection we prove the following proposition, which is crucial in order to obtain the sharp dichotomy phenomenon as stated in our main results Theorem \ref{th:regime_sharp} and Theorem \ref{th:critspeed}:
\begin{Prop}\label{prop:conv}
Let $u_0$ be any nonnegative, bounded and compactly supported initial datum. Then the associated solution $u(t,z)$ of \eqref{problemmonobis} either spreads, goes extinct, or is grounding in the sense of Definition~\ref{def:main}.
\end{Prop}
Let us remind the reader that by grounding we may only mean uniform convergence to the critical ground state~$p_{\alpha^*_c}$. We will still assume in this subsection that $c \in [2 \sqrt{g'(0)},c^*]$, even though we do not state it explicitly in the above proposition: the reason is that Proposition~\ref{prop:conv} clearly holds true for any $c \in [ 0 , 2 \sqrt{g'(0)}) \cup (c^*, +\infty)$, as immediately follows from Theorems~\ref{th:spread_slow} and~\ref{th:nospread}. 

Note also that the critical ground state $p_{\alpha^*_c}$ does not exist when $c < 2\sqrt{g'(0)}$ or $c \geq c^*$, in which case the only possible outcomes are spreading and extinction. In particular, when $c = c^*$, the combination of Proposition~\ref{prop:conv} and Theorem~\ref{th:nospread} implies that extinction always occurs, which completes the proof of Theorem~\ref{th:regime}.
\begin{proof}
We only consider initial data $u_0$ such that spreading does not occur. Thus Proposition~\ref{th:boundexp} applies, along with the lemmas involved in its proof. In particular, it follows from Lemma~\ref{uepsilonthm} and the simple fact that $Ae^{\mu_c z}$ is a supersolution of \eqref{problemmonobis} for all $z \leq 0$ (where $\mu_c >0$ was introduced in Section~\ref{sec:stationarysol_sub1} and $A>0$ can be chosen arbitrarily large), that for all $\varepsilon >0$, we have up to increasing $Z_\varepsilon$:
$$\forall t \geq 0 , \ \forall |z | \geq Z_\varepsilon, \quad u(t,z) \leq \varepsilon .$$
Since $p_{\alpha^*_c} (z) \to 0$ as $z \to \pm \infty$, it is straightforward that we only need to prove that $u (t,z)$ converges locally uniformly with respect to $z$ to either 0 or~$p_{\alpha^*_c}$ as $t \to +\infty$.

Let us first introduce the corresponding $\omega$-limit set:
\be\label{omegalimitset}\Omega(u_0)=\underset{t>0}{\cap}\overline{\left\{u(\tau,\cdot),\: \tau\geq t\right\}},\ee
with $u$ the solution of \eqref{problemmonobis} such that $u(0,z)=u_0(z)$ for all $z\in\R$. Here the closure in~\eqref{omegalimitset} is taken in the locally uniform topology. Since the set $\{ t \geq 1 \, | \ u(t,\cdot ) \}$ is relatively compact with respect to the locally uniform topology, the $\omega$-limit set is not empty and our goal now rewrites as proving that $\Omega (u_0)$ is reduced to a singleton which is either $\{0\}$ or $\{ p_{\alpha^*_c}\}$.

Before we proceed, let us also mention the well-known fact that the relative compactness of the set $\{ t \geq 1 \, | \ u(t,\cdot ) \}$ implies that $\Omega (u_0)$ is connected. This will prove useful below.\\

Assume now that $c \in [ 2 \sqrt{g'(0)}, c^* ]$. We will prove that the $\omega$-limit set consists only of stationary solutions of~\eqref{problemmonobis}, and follow an argument of Du and Matano~\cite{DM}. 

Choose some $w\in\Omega(u_0)$, and $(t_n)_n$ such that $t_n\to+\infty$ and $u(t_n,\cdot)\to w(\cdot)$ locally uniformly as $n \to +\infty$. Here we will assume that $w \not \equiv 0$ and prove that $w \equiv p_{\alpha^*_c}$ (which is a contradiction when $c =c^*$). Thanks to standard parabolic estimates and up to extraction of a subsequence, the sequence $u_n (t,z)=u(t+t_n,z)$ converges locally uniformly with respect to both $t$ and $z$ to some nonnegative solution $u_\infty$ of~\eqref{problemmonobis} such that $u_\infty (0,z) = w (z)$ for all $z \in \R$. In particular, by the strong maximum principle and from our choice of $w \not \equiv 0$, we have that $w (z) >0$ for all $z \in \R$.

Consider now $p$ the solution of the ODE
$$p ''+c p'+f(z,p)=0,$$
such that $p(0)=u_\infty (0,0) = w (0) >0$ and $p '(0)=\partial_z u_\infty (0,0) = w ' (0)$. 
Let us denote
$$z_1 = \inf \{  z < 0 \, | \ p (z ) > 0 \} \in [-\infty,0),$$
and 
$$z_2 = \sup \{  z > 0 \, | \ p (z ) > 0 \} \in (0 , +\infty].$$
Note that $z_1$ (respectively $z_2$) may be infinite if $p$ does not change sign on $\R_-$ (respectively $\R_+$).

Since clearly $u_0 \not \equiv p$, one can define the number of sign changes of $u(t,\cdot) - p ( \cdot)$ on $I$ the closure of $(z_1,z_2)$, and claim that is finite for all $t >0$. As before we denote by $\mathcal{Z}_I [ u(t,\cdot)- p (\cdot)]$ the number of sign changes of $u(t,\cdot)- p (\cdot)$ in $I$. \begin{claim}\label{claim:Z1}
For all $t >0$,
$$\mathcal{Z}_I [ u(t,\cdot)- p (\cdot)] < +\infty.$$
\end{claim}
\begin{proof}[Proof of Claim~\ref{claim:Z1}] Let us first consider the case when $I$ is a bounded interval. As $u(t,z_1)- p (z_1 )>0$ and $u(t,z_2)- p (z_2)>0$ for all $t>0$, and since $u(t,z)- p (z )$ satisfies a linear parabolic equation on $(z_1,z_2)$ in a similar fashion as in the proof of Lemma~\ref{lem:energy3}, we know from Proposition~\ref{prop:chang_signR} that the number of zeros is finite which proves the claim. 

Next assume that $I = \R$. Since $f (z,s) = -s$ for all $z < 0$ and $s \in \R$, it is straightforward that
$$\forall z \leq 0 , \quad p (z) = A e^{\mu_c z} + B e^{-\nu_c z},$$
where $\mu_c$ is defined in Section~\ref{sec:stationarysol}, $\nu_c := \frac{c + \sqrt{c^2 +4}}{2} >0$, and either $B>0$ or $B=0 <A$. On the other hand, from a standard phase plane analysis of~\eqref{ODE2} (see the proof of Theorem~\ref{th:stationary} and Proposition~\ref{prop:exp_decay}), we have either $\liminf p (z) \geq 1$ as $z \to +\infty$, or $p (z) \to 0$ as $z \to +\infty$ at an exponential rate that is
$$\limsup \left| \frac{p ' (z)}{p (z)} \right| \leq  \frac{c + \sqrt{c^2 - 4 g'(0)}}{2} \quad \mbox{ as } \ z \to +\infty .$$
In any case, using the fact that $u_0$ has compact support, it is straightforward (see again the supersolution~$\overline{u}_2$ in Section~\ref{sec:large_c}) that for any $t>0$, there exists some $Z >0$ such that
$$\forall |z| \geq Z, \quad u(t,z) < p (z).$$
Using Proposition \ref{prop:chang_signR} and the fact that zeros of $u- p$ do not accumulate, we again reach the wanted conclusion.

The remaining cases $-\infty < z_1 < z_2 = +\infty$ and $-\infty = z_1 < z_2 < + \infty$ easily follow from the same arguments and we omit the details.
\end{proof}

Let us go back to the proof of Proposition~\ref{prop:conv}. We now prove that $u_\infty \equiv p$. By Proposition \ref{prop:chang_signR}, $\mathcal{Z}_I [ u(t,\cdot) - p (\cdot )]$ is also nonincreasing with respect to time, which implies that it is constant for large times and using again Proposition \ref{prop:chang_signR}, $u(t,\cdot)-p(\cdot)$ has only simple zeros on $I$ for large time $t$. One may then apply Lemma~2.6 from~\cite{DM} to reach the wanted conclusion. However, we include the argument for the sake of completeness.

From our choice of $p$ such that $p (0) = u_\infty (0,0)$ and assuming by contradiction that $u_\infty \not \equiv p$, we can also apply Proposition \ref{prop:chang_signR} to get that the zeros of $u_\infty (0,\cdot) - p(\cdot)$ do not accumulate. In particular, there exist $\tau>0$ and $\epsilon>0$ such that $[-\varepsilon,\varepsilon] \subset I$ and 
$$u_\infty(t,\pm\epsilon)\neq 0,\quad \forall\:t\in[-\tau,\tau].$$
Moreover, as $u_\infty (0,\cdot) - p (\cdot)$ has a degenerate zero at $0$,
$$\mathcal{Z}_{[-\varepsilon,\varepsilon] }[u_\infty(-\tau,\cdot)-p(\cdot)]>\mathcal{Z}_{[-\varepsilon,\varepsilon] } [u_\infty(\tau,\cdot)-p(\cdot)].$$
Because degenerate zeros may only appear at discrete times, one may also assume up to reducing $\tau$ that the zeros of $u_\infty(\pm\tau,\cdot)-p(\cdot)$ are all simple in $[ - \varepsilon, \varepsilon]$. Besides, by standard parabolic estimates, the sequence $u_n  (t,z) = u(t+t_n,z)$ and its spatial derivative $\partial_z u_n (t,z)$ converge locally uniformly to $u_\infty (t,z)$ and $\partial_z u_\infty (t,z)$ respectively. It follows that
$$\mathcal{Z}_{[-\varepsilon,\varepsilon]}[u_n(-\tau,z)-p(z)]>\mathcal{Z}_{[-\varepsilon,\varepsilon]}[u_n(\tau,z)-p(z)]$$
and 
$$u_n(t,\pm\epsilon)\neq 0,\quad \forall\:t\in[-\tau,\tau],$$
for $n$ large enough. This implies that, for any large $n$, $u_n(t,\cdot)-p(\cdot)$ has a degenerate zero in the interval $(-\varepsilon,\varepsilon)$ for some time $t \in (-\tau,\tau)$. However, as explained above, Claim~\ref{claim:Z1} and Proposition~\ref{prop:chang_signR} imply that $u (t,\cdot) - p (\cdot)$ has only simple zeros for large times, which is a contradiction. We conclude that $u_\infty \equiv p$, and it immediately follows that $w (\cdot ) \equiv u_\infty (0,\cdot)$ is a (bounded and positive) stationary solution.

Now recall that $u$ satisfies the exponential bound from Proposition~\ref{th:boundexp} when~$c=2\sqrt{g'(0)}$ or $c\in(2\sqrt{g'(0)},c^*]$, and so does $w$ by passing to the limit as $t \to +\infty$. On the other hand, it follows from Proposition~\ref{prop:exp_decay} that $p_{\alpha^*_c}$ (when it exists) is the only bounded and positive stationary solution satisfying the same inequality, so that $w \equiv p_{\alpha^*_c}$. Therefore, when $c = c^*$, we have reached another contradiction (in this case by Theorem~\ref{th:stationary} there is no critical ground state) and conclude that $\Omega (u_0) = \{0\}$. If $c\in[2 \sqrt{g'(0)} ,c^*)$, then $\Omega (u_0) \subset \{0,p_{\alpha^*_c} \}$ and, since it is connected, either $\Omega (u_0) = \{0\}$ or $\Omega (u_0) = \{p_{\alpha^*_c} \}$. Proposition~\ref{prop:conv} is proved.\end{proof}

\section{Sharp transitions phenomena}\label{sec:sharpf}

We have already shown that the solution of \eqref{problemmonobis} either spreads, goes extinct, or converges uniformly to the critical ground state $p_{\alpha^*_c}$, see Proposition~\ref{prop:conv} above. In this section, we will prove Theorems~\ref{th:regime_sharp} and~\ref{th:critspeed}. Both theorems highlight the fact that, whether we fix the speed while the initial datum varies, or whether we fix the initial datum while the speed varies, there is a sharp transition from spreading behaviour to extinction behaviour.

\subsection{Sharp dichotomy}\label{sec:sharpdich}

Let us first prove Theorem~\ref{th:regime_sharp}. Here we assume that $2 \sqrt{g'(0)} < c^*$ and fix $c \in [2 \sqrt{g'(0)},c^*)$. We also introduce a family $(u_{0,\sigma})_{\sigma > 0}$ of nonnegative, compactly supported and bounded initial data such that
$$\forall \sigma ' > \sigma, \quad u_{0,\sigma'} \geq  u_{0,\sigma} \ \mbox{ and } \ u_{0,\sigma '} \not \equiv u_{0,\sigma},$$
and
$$\forall \sigma > 0, \quad \sigma ' \to \sigma \Rightarrow \| u_{0,\sigma '} - u_{0,\sigma} \|_{L^1 (\R)} \to 0.$$
We also denote, for each $\sigma > 0$, by $u_\sigma$ the solution of \eqref{problemmonobis} with initial datum $u_{0,\sigma}$. Then define the (possibly empty) sets
$$\Sigma_0=\{\sigma > 0 \, | \  u_\sigma(t,z)\to0, \text{ as } t\to+\infty \text{ uniformly in }z\},$$
and
$$\Sigma_1=\{\sigma > 0 \, | \ u_\sigma(t,z)\to p_+^c(z), \text{ as } t\to+\infty \text{ locally uniformly in }z\}.$$
We also define
$$\sigma_* := \left\{ \begin{array}{l}
\sup \Sigma_0 ,  \vspace{3pt}\\
0 \mbox{ if } \Sigma_0 = \emptyset ,
\end{array}\right.
\quad \sigma^* := \left\{ \begin{array}{l}
\inf \Sigma_1 ,  \vspace{3pt}\\
+\infty \mbox{ if } \Sigma_1 = \emptyset .
\end{array}\right.$$
By the comparison principle, we have that:
$$\forall t \geq 0 , \ \forall z \in \R , \quad \sigma' > \sigma \Rightarrow u_{\sigma '} (t,z) \geq u_\sigma (t,z).$$
In particular, it is straightforward that $ \sigma_* \leq \sigma^*$ and
$$(0,\sigma_* ) \subset \Sigma_0 \subset (0,\sigma_*], \quad (\sigma^* ,+\infty) \subset \Sigma_1 \subset [\sigma^*, +\infty).$$
\begin{claim}\label{claim_sigma1}
The sets $\Sigma_0$ and $\Sigma_1$ are open. In particular, if $\sigma \in [\sigma_*, \sigma^*] \setminus (\{0\}\cup\{+\infty\})$, then $u_\sigma (t,z)$ converges as $t \to +\infty$ to $p_{\alpha^*_c} (z)$ uniformly with respect to $z$.
\end{claim}
\begin{proof}[Proof of Claim~\ref{claim_sigma1}]
Let us first take $\sigma^* \in \Sigma_1$, and prove that for any $\delta >0$ small enough, then $\sigma^* - \delta \in \Sigma_1$. Let also some compactly supported initial datum $0 \leq u_{0,1} < p_+^c$ be such that the associated solution spreads, using Theorem~\ref{thm:interm_0}. 
Since $u_{\sigma^*} (t,z)$ converges locally uniformly to $p_+^c (z)$ as $t \to +\infty$, and by continuity of the solution $u_\sigma$ with respect to $\sigma$ in the locally uniform topology, it is straightforward that one cand find $T>0$ large enough and $\delta_0 >0$ such that, for all $0 < \delta \leq \delta_0$:
$$\forall z \in \R , \quad u_{\sigma^* - \delta} (T,z) \geq u_{0,1} (z).$$
Applying the comparison principle and recalling Lemma~\ref{lem:sup_cv}, we infer that $[\sigma^* - \delta_0, \sigma^*] \subset \Sigma_1$.

Next, we take $\sigma_* \in \Sigma_0$ and prove that for $\delta>0$ small enough, then $\sigma_* + \delta \in \Sigma_0$. 
Here the difficulty lies in the fact that (unlike in the previous case or in the related work~\cite{DM}), the trivial state 0 is only stable with respect to perturbations which decay fast enough at infinity. Thus, we need to use the fact that the initial data $u_{0,\sigma}$ have compact support, which we do through the application of Proposition~\ref{th:boundexp}.

First consider the case when $ 2 \sqrt{g'(0)} < c < c^*$. Then there exists $0 < \epsilon < \alpha^*_c$ small enough so that, for all $0 \leq s \leq \epsilon$,
$$g'(s) \leq \frac{c^2}{4}.$$
It easily follows that the function $\overline{u}_1 (z) = \epsilon e^{-\frac{c}{2} z}$ satisfies, for all $z \geq 0$,
$$\overline{u}_1 '' + c \overline{u}_1 ' + g(\overline{u}_1) \leq 0.$$
Moreover, for $z \leq 0$, we also have that
$$\overline{u}_1 '' + c \overline{u}_1 ' - \overline{u}_1 \leq 0.$$
In other words, $\overline{u}_1$ is a supersolution of \eqref{problemmonobis}. Recall that the solution of \eqref{problemmonobis} is bounded for any bounded and nonnegative initial datum, and therefore one can fix $M > \|u_{\sigma_* +1}\|_{L^\infty (\R_+ \times \R)}$. Now let $Z$ be large enough so that $\overline{u}_1 (-Z) \geq M$, and so that applying Proposition~\ref{th:boundexp} where $\lambda_\gamma (c) > \frac{c}{2}$:
$$\forall t \geq 0, \ \forall |z| \geq Z, \quad u_{\sigma_*} (t,z) \leq \frac{\overline{u}_1 (z)}{2}.$$
Then, since we assumed that $\sigma_* \in \Sigma_0$, there exists some $T>0$ such that
$$\forall |z| \leq Z , \quad u_{\sigma_*} (T,z) \leq \frac{\overline{u}_1 (z)}{2}.$$
We now claim that $u_{\sigma_* + \delta} (T,\cdot ) \leq \overline{u}_1 (\cdot)$ for any small enough $\delta \in (0,1)$. When $z \leq -Z$, the inequality simply follows from our choice of $Z$ and $M$ above. Then, by continuity (in the locally uniform topology) of solutions of \eqref{problemmonobis} with respect to the initial datum, for any $\delta \in (0,1)$ small enough we have that
$$\forall | z | \leq Z, \quad u_{\sigma_* + \delta} (T,z) \leq \overline{u}_1 (z),$$
but also that
\begin{equation}\label{simple1}
\forall 0 \leq t \leq T , \ u_{\sigma_* + \delta} (t,Z) \leq \overline{u}_1 (Z).
\end{equation}
Without loss of generality, one may increase $Z$ so that the support of $u_{0,\sigma_* + 1}$ is included in $(-Z,Z)$: in particular, for any $ \delta \in (0,1)$, the support of $u_{0,\sigma_* + \delta}$ is also included in $(-Z,Z)$ and $u_{0,\sigma_*+ \delta} (z) \leq \overline{u}_1 (z)$ for all $z \geq Z$. Together with~\eqref{simple1} and applying the comparison principle on the domain $(0,T) \times (Z,+\infty)$, this implies that for any $\delta \in (0,1)$ small enough,
$$\forall z \geq Z , \quad u_{\sigma_* + \delta} (T,z) \leq \overline{u}_1 (z).$$
As announced we have obtained that $u_{\sigma_* + \delta} (T,z) \leq \overline{u}_1 (z)$ for all $z \in \R$, and applying again the comparison principle the same inequality holds for all times larger than $T$. Finally, recalling that $p_{\alpha^*_c} (0) = \alpha^*_c > \epsilon = \overline{u}_1 (0)$, it follows from Proposition~\ref{prop:conv} that $u_{\sigma_* +\delta} (t,\cdot)$ converges uniformly to 0 as $t \to +\infty$, i.e. $\sigma_* + \delta \in \Sigma_0$.

Now assume that $c = 2\sqrt{g'(0)} < c^*$. Fix $M> \|u_{\sigma_* +1}\|_{L^\infty (\R_+ \times \R)}$ as above. By a similar computation as in the proof of Proposition~\ref{th:boundexp}, the function 
$$\overline{u}_2 (z) := M (1+ z^{3/4}) e^{-\frac{c}{2} z}$$
satisfies
$$\overline{u}_2 '' + c \overline{u}_2 ' + g (\overline{u}_2) \leq 0$$
for all $z \geq Z_0$ where $Z_0 >0$ is large enough (depending only on $M$ and $g$), as well as
$$\overline{u}_2 '' + c \overline{u}_2 ' - \overline{u}_2 \leq 0$$
for all $z >0$. It follows that $\overline{u}_2 (z + Z_0)$ is a supersolution of \eqref{problemmonobis} on the half line $\{ z > - Z_0\}$. Then, proceeding similarly as above, using Proposition \ref{th:boundexp}, one can find some $T >0$ so that, for any small enough $\delta >0$,
$$\forall z \geq - Z_0, \quad u_{\sigma_* + \delta} (T,z) \leq \overline{u}_2 (z + Z_0).$$
Then, we apply the comparison principle on $(T,+\infty) \times (-Z_0,+\infty)$, thanks to the fact that $u_{\sigma_* + \delta} (t,-Z_0) \leq M = \overline{u}_2 (0)$ for all $t >0$, and get that 
$$\forall t \geq T ,\ \forall z \geq - Z_0, \quad u_{\sigma_* + \delta} (t,z) \leq \overline{u}_2 (z + Z_0).$$
Since $\overline{u}_2 (z) \to 0$ as $z \to +\infty$, we may increase $Z_0$ without loss of generality so that $\overline{u}_2 (Z_0) < \alpha^*_c = p_{\alpha^*_c} (0)$. Finally, it follows from the above inequality and Proposition~\ref{prop:conv} that $u_{\sigma_* + \delta}$ goes extinct as time goes to infinity. Thus $\sigma_* + \delta \in \Sigma_0$ for any small enough $\delta$.

We have now proved that $\Sigma_0$ and $\Sigma_1$ are open, and hence $[\sigma_*, \sigma^*] \cap (\Sigma_0 \cup  \Sigma_1) = \emptyset $. It immediately follows from Proposition~\ref{prop:conv} that for any real number $\sigma \in [\sigma_* ,\sigma^*] \setminus (\{0\}\cup\{+\infty\})$, the solution $u_\sigma$ converges uniformly to $p_{\alpha^*_c}$ as time goes to~$+\infty$.
\end{proof}
The sharpness of the threshold between extinction and spreading, namely the fact that $\sigma_* = \sigma^*$, and hence Theorem~\ref{th:regime_sharp}, immediately follows from Claim~\ref{claim_sigma1} and the next lemma:
\begin{Lemma}\label{lem:lambdacritground}
If there exists $\sigma^*>0$ such that $u_{\sigma^*}(t,z)\to p_{\alpha^*_c}(z)$ as $t\to+\infty$ uniformly in $z$, then for all $\sigma<\sigma^*$, $u_\sigma$ vanishes.
\end{Lemma}
\begin{proof}
This proof is largely inspired by that of Lemmas 4.4 and 4.5 in \cite{DM}. We first prove that for any $\sigma_1 < \sigma_2$, there exists $t_0$,	 $\delta$ and $\epsilon$ positive constants such that
\begin{equation}\label{eq:DM1}
u_{\sigma_1}(t,z)<u_{\sigma_2}(t+\delta,z+a), \quad \forall \:t\geq t_0,\:z\in\R, \: 0<a<\epsilon.
\end{equation}
Using the same argument as \cite{DM} [Proposition~1.8] and the fact that $s\mapsto f(z,s)$ is Lipschitz-continuous, uniformly with respect to $z\in\R$, one can show that there exists $t_0 >0$ and $K >0$ such that the supports of $u_{0,\sigma_1}$ and $u_{0,\sigma_2}$ are included in $(-K+1, K-1)$, and such that $u_{\sigma_1}$ and $u_{\sigma_2}$ are both increasing in $t$ for all $t \in (0,2t_0)$ and $|z| \geq K$. We know from the strong maximum principle that $u_{\sigma_1}(t,z)<u_{\sigma_2}(t,z)$ for all $t>0$, $z\in\R$, and thus by the $C^1$-regularity of solutions of \eqref{problemmonobis} for positive times, one can check that for any $0 < \epsilon <1$ and $\delta>0$ small enough,
\be\label{inequ1}
\forall |z|<K, \: 0<a<\epsilon , \quad u_{\sigma_1}(t_0 ,z)<u_{\sigma_2}(t_0 +\delta,z+a)  .\ee
Let us now show the same inequality for $|z| \geq K$. Using the fact that $u_{\sigma_2}$ is increasing in $t$ for $t \in (0,2 t_0)$ and $|z| \geq K$, and assuming that $\delta < t_0$ without loss of generality, we get that 
$$ \forall t\in[0,t_0], \:z=\pm K, \quad u_{\sigma_1}(t,z)<u_{\sigma_2}(t+\delta,z).$$ 
Thanks to the continuity of the solutions up to time $t=0$ at $z = \pm K$ (which lie outside of the initial support), we get up to reducing $\varepsilon$ that 
$$\forall t\in[0,t_0], \:z=\pm K,\ 0<a<\epsilon , \quad u_{\sigma_1}(t,z)<u_{\sigma_2}(t+\delta,z+a) .$$
Moreover, recalling that the support of $u_{0,\sigma_1}$ is included in $(-K+1,K-1)$, we get that
$$\forall | z | \geq K , \ 0<a < \epsilon , \quad u_{\sigma_1}(0,z)= 0 <u_{\sigma_2}(\delta,z+a).$$
Then using the comparison principle on each interval $(K,+\infty)$ and $(-\infty,-K)$, one has that
$$\forall |z | \geq K , \ 0<a < \epsilon , \quad u_{\sigma_1}(t_0,z)<u_{\sigma_2}(t_0+\delta,z+a).$$
Putting this together with~\eqref{inequ1}, we have that $u_{\sigma_1}(t_0,z)<u_{\sigma_2}(t_0+\delta,z+a)$  for all $z\in\R$. Then, using the comparison principle and the fact that $f(z,u)\leq f(z+a,u)$ for all $a>0$, we get for all $0< a < \varepsilon$, $t \geq t_0$, $z\in\R$ that 
$$u_{\sigma_1}(t,z)<u_{\sigma_2}(t+\delta,z+a).$$
We have now proved~\eqref{eq:DM1}. Now under the assumptions of Lemma~\ref{lem:lambdacritground}, choose $\sigma_2=\sigma^*$ and $\sigma_1=\sigma<\sigma^*$, and infer that there is some $a>0$ such that, for all $z \in \R$,
$$\underset{t\to+\infty}{\lim}u_{\sigma}(t,z)\leq p_{\alpha_c^*}(z+a).$$
Together with Proposition~\ref{prop:conv}, this easily implies that $u_\sigma$ converges uniformly to 0 as time goes to infinity and the lemma is proved.
\end{proof}

\subsection{Sharp speed}\label{sec:critspeed}
In this section we turn to the proof of Theorem \ref{th:critspeed} and show that for a fixed initial condition~$u_0$ there exists a sharp speed $c(u_0)\in[2\sqrt{g'(0)},c^*]$ such that spreading occurs when $c<c(u_0)$ and vanishing occurs when $c>c(u_0)$.

\begin{proof}[Proof of Theorem \ref{th:critspeed}]
Assume that the initial condition $u_0$ is bounded, nonnegative, non trivial and compactly supported. From Theorem~\ref{th:regime} we already know that for all $c<2\sqrt{g'(0)}$, the solution in the moving frame converges to $p^c_+$ locally uniformly as time goes to infinity, and that for all $c\geq c^*$ the solution converges to 0 uniformly. From this we can define two thresholds $2\sqrt{g'(0)}\leq\underline{c}\leq\overline{c}\leq c^*$ by
\begin{eqnarray*}
\underline{c} & :=& \max \{ c \geq 0 \, | \ \forall c' < c , \ u(t,z) \to p_+^c (z) \mbox{ as } t \to +\infty \mbox{ locally uniformly, i.e. spreading occurs} \}, \vspace{3pt}\\
\overline{c} & := & \min \{ c \geq 0 \, | \ \forall c' > c , \ u(t,z) \to 0 \mbox{ as } t \to +\infty \mbox{ uniformly, i.e. extinction occurs} \}.
\end{eqnarray*}
In addition, one can notice that the solution $u(t,x)$ in the non moving frame is decreasing with respect to~$c\geq 0$. Indeed if $u_{c_1}$, respectively $u_{c_2}$, are solutions of the original problem \eqref{problemmono}  with $c= c_1$, respectively $c= c_2$, such that $c_1<c_2$ and $u_{c_1}(0,x)=u_{c_2}(0,x) = u_0 (x)$ for all $x\in\R$, one can conclude using the parabolic maximum principle and the monotonicity of $f(z,u)$ with respect to $z$ that for all $t>0$, $x\in\R$,
$$u_{c_1}(t,x)>u_{c_2} (t,x).$$
With the same abuse of notations as before, we denote by $u_{c_1} (t,z)$ and $u_{c_2} (t,z)$ the solutions of \eqref{problemmonobis} with respectively $c = c_1$ and $c=c_2$ (note however that, since $c_1 \neq c_2$, both functions are derived from different change of variables). The inequality above then rewrites as
$$\forall t \geq 0 , \ \forall z \in \R , \quad u_{c_2} (t,z ) < u_{c_1} (t, z+ (c_2 -  c_1) t ).$$
It immediately follows that, if $u_{c_1} (t,z)$ goes extinct in the sense of Definition~\ref{def:main} and converges uniformly to~0 as $t\to +\infty$, then $u_{c_2} (t,z)$ also goes extinct as $t \to +\infty$. Similarly, if $u_{c_1} (t,z)$ converges uniformly to $p_{\alpha^*_{c_1}} (z)$, then the above inequality implies that $u_{c_2} (t,z)$ converges locally uniformly to 0, hence uniformly by Proposition~\ref{prop:conv}. Therefore, recalling from Proposition~\ref{prop:conv} that the solution either goes extinct, spreads or converges uniformly to the critical ground state, one can already conclude that $\underline{c} = \overline{c}$. From now on, we denote this speed by $c (u_0)$. It only remains to investigate how the solution of \eqref{problemmonobis} behaves when $c = c (u_0)$. 

Let us first prove that, when $c = c(u_0)$, then either grounding or extinction occurs, in the sense of Definition~\ref{def:main} (note that part~$(iii)(b)$ of Theorem~\ref{th:critspeed} immediately follows). Proceed by contradiction and assume that $u_{c (u_0)} (t,z)$ the solution of \eqref{problemmonobis} with $c=c (u_0)$ spreads and converges locally uniformly to $p^+_{c (u_0)} (z)$. In particular, by Theorem~\ref{th:regime}, $c(u_0)<c^*$. Using the fact that $p^+_{c (u_0)} (z) \to 1$ as $z \to +\infty$, then for any $D>0$ and $\epsilon >0$, one can find some $Z>0$ and $T>0$ large enough such that
$$u_{c (u_0)} (T,z) \geq (1 -\epsilon) \chi_{(-D,D)} (z-Z),$$
where $\chi$ denotes the usual characteristic function. By standard estimates, one can show that the solution depends continuously, in the locally uniform topology, on the parameter $c$.
In particular, for any $c'$ close enough to $c (u_0)$, the solution $u_{c'}$ of \eqref{problemmonobis} with $c = c'$ also satisfies
$$u_{c'} (T,z) \geq (1 -2 \epsilon) \chi_{(-D,D)} (z-Z).$$
Now, by Theorem~\ref{thm:interm_0}, there exists some compactly supported initial initial datum $u_{0,1} <1$ such that the associated solution of \eqref{problemmonobis} with $c = \frac{c(u_0) + c^*}{2}  \in ( c(u_0),c^*)$ spreads. As the above discussion on the monotonicity of solutions with respect to~$c$ also applies to the solution with initial datum $u_{0,1}$, we have that the solution of \eqref{problemmonobis} with initial datum $u_{0,1}$ also spreads for any $c \in (c (u_0) , \frac{c(u_0) + c^*}{2}   )$. Now choose $D >0$ large enough and $\epsilon >0$ small enough so that
$$u_{0,1} (z) < (1 - 2 \epsilon) \chi{(-D,D)} (z).$$
Then, for any $c' > c (u_0)$ but close enough, we have
$$u_{c'} (T,z) \geq u_{0,1} (z-Z).$$
Recalling that $Z>0$ and thanks to the monotonicity of $f$ with respect to its first variable, the latter inequality and Proposition~\ref{prop:conv} imply that $u_{c'}$ spreads, which contradicts the definition of $c(u_0)$.

It only remains to show that, when $c= c(u_0)> 2\sqrt{g'(0)}$, then extinction may not occur. Proceed by contradiction and assume that $u_{c(u_0)} (t,\cdot)$ converges uniformly to 0 as $t \to +\infty$. Let us show that extinction also occurs for $c' < c (u_0)$ but close enough, contradicting the definition of~$c(u_0)$. The argument will be similar to the proof of the openness of $\Sigma_0$ in Claim~\ref{claim_sigma1}. Let $\delta >0$ be small enough so that $ c(u_0) - \delta > 2 \sqrt{g'(0)}$ and, as in the proof of Claim~\ref{claim_sigma1}, choose $\eta \in (0, \alpha^*_{c (u_0) - \delta})$ small enough so that the function
$$\overline{u} (z) :=  \eta e^{-\frac{c(u_0)-\delta}{2} z}$$
satisfies
$$\overline{u} '' + (c (u_0) - \delta) \overline{u} ' + f(z,\overline{u}) \leq 0 ,$$
for all $z \in \R$. As $\overline{u}$ is a decreasing function, it immediately follows that it is a supersolution of~\eqref{problemmonobis} for any $c' \in [c(u_0) - \delta, c(u_0)]$. Moreover, using again the monotonicity of solutions of~\eqref{problemmono} (in the non moving frame) with respect to~$c$, and the uniform boundedness of the solution for any bounded initial datum, there exists $M>0$ such that for all $c' \in [c (u_0) -\delta, c (u_0)]$,
$$\|u_{c'} \|_{L^\infty (\R_+ \times \R)} \leq \|u_{c (u_0) - \delta}\|_{L^\infty (\R_+ \times \R)} < M.$$
Next, let $Z >0$ be such that the support of $u_0$ is included in $(-Z,Z)$, that $\overline{u} (z) \geq M$ for any $z \leq -Z$ as well as, by Proposition~\ref{th:boundexp}, 
$$\forall t \geq 0, \ \forall z \geq Z , \quad u_{c (u_0)} (t,z) \leq \frac{\overline{u} (z)}{2}.$$
Because $u_{c (u_0)}$ goes extinct as $t \to +\infty$, there also exists $T>0$ such that
$$\forall | z | \leq Z , \quad u_{c (u_0)} (T,z) \leq \frac{\overline{u} (z)}{2}.$$
Up to reducing $\delta$ and by continuity of solutions of \eqref{problemmonobis} with respect to $c$ in the locally uniform topology, we get for any $c' \in [ c (u_0)  - \delta, c (u_0)]$ that
$$\forall 0 \leq t \leq T,  \quad u_{c' } (t, Z) \leq \overline{u} (Z),$$
and 
$$\forall | z | \leq Z , \quad u_{c '} (T,z) \leq \overline{u} (z).$$
Applying a comparison principle on $(0,T) \times (Z,+\infty)$, and since $\overline{u} (z) \geq \|u_{c (u_0)-\delta} \|_{L^\infty (\R_+ \times \R)}$ for all $z \leq -Z$, one can check that for all $c' \in[c(u_0)-\delta,c(u_0)]$,
$$\forall z \in \R ,\quad u_{c'} (T,z) \leq \overline{u} (z).$$
Thus $u_{c'} (t,z) \leq \overline{u} (z)$ for all $z \in \R$ and $t \geq T$. Proposition~\ref{prop:conv}, together with the fact that $\overline{u} (0) = \eta < \alpha^*_{c (u_0) - \delta} \leq \alpha^*_{c(u_0)}$ (from Theorem \ref{th:stationary}), implies that $u_{c'} (t,z)$ converges uniformly in $z$ to 0 as $t \to +\infty$.  We have reached the wanted contradiction, and we conclude as announced that, when $c = c (u_0) > 2 \sqrt{g'(0)}$, then grounding occurs. In particular, from Theorem~\ref{th:regime} necessarily $c < c^*$, and part~$(a)$ of Theorem~\ref{th:critspeed} is proved. Noting that part~$(c)$ of Theorem~\ref{th:critspeed} simply follows from Theorem~\ref{th:regime}, this completes the proof.
\end{proof}

\begin{Rk}
Let us briefly check that none of the remaining possibilities may be ruled out: more precisely, for any $c \in [2 \sqrt{g'(0)},c^*)$, there exists an initial datum such that $c = c (u_0)$ and grounding occurs, as well as some initial datum such that $c (u_0) = 2 \sqrt{g'(0)}$ and extinction occurs at the threshold speed. Indeed, it suffices to take for any $c \in [2 \sqrt{g'(0)} , c^*)$, thanks to Theorems~\ref{th:regime} and~\ref{th:regime_sharp}, an initial datum such that grounding occurs (or, if $c =2 \sqrt{g'(0)}$, such that extinction occurs), and to observe a posteriori from Theorem~\ref{th:critspeed} that necessarily $c (u_0) = c$.
\end{Rk}

\appendix

\section{Appendix}\label{A:number_zero}

In this appendix we briefly state some properties on the number of sign changes and/or the number of zeros of the solution of a semilinear scalar parabolic equation, which we use extensively in our proofs. All of these properties are contained or easy consequences of \cite{A,DM,matano}. 

Let $h$ be a solution of the following parabolic equation
\be\label{eq:parabolic_zeros}
\partial_t h-\partial_{zz}h-c\partial_z h-b(t,z)h=0, \quad \forall t\in(0,+\infty), \:z\in I\ee
where $I$ is any interval, and the coefficient $b$ and the solution $h$ are bounded. We denote, for any time $t >0$ and provided that $h (t,\cdot) \not \equiv 0$, by $\mathcal{Z}_I [h (t,\cdot)] $ the number of sign changes of $h (t,\cdot)$. More precisely, $\mathcal{Z}_I [h (t\cdot)] = 0$ if either $h (t,\cdot) >0$ or $h (t,\cdot) < 0$, and otherwise it is defined as the supremum over all integers~$k$ such that there exists $z_1 < z_2 < ... < z_k$ in $I$ with
$$\forall i= 1, ... , k-1 , \quad h (t,z_i) . h (t , z_{i+1}) < 0.$$
Note that, when all the zeros of $h(t,\cdot)$ are simple, this definition clearly coincides with the number of zeros of $h (t,\cdot)$.\\

The main property of the number of sign changes of a solution of a semilinear one-dimensional parabolic equation is the following Sturmian principle, which extends Lemma~2.3 in~\cite{DM} which dealt with the particular case $c=0$, and similarly follows from~\cite{A,matano}:
\begin{Prop}\label{prop:chang_signR}
Let $h\not\equiv0$ be a solution of \eqref{eq:parabolic_zeros} which never vanishes on $\partial I$ the (possibly empty) boundary of $I$. Then for each $t\in\R^*_+$ the zeros of the function $h$ do not accumulate in $I$, and
\begin{enumerate}[(i)]
\item $t\mapsto Z_I [h(t,\cdot)]$ is nonincreasing;
\item if there exists $t_0>0$, $z_0\in I$ such that $h(t_0,z_0)=h_z(t_0,z_0)=0$, then 
$$Z_I  [h(t,\cdot)]>Z_I [h(s,\cdot)], \quad \forall \:0<t<t_0<s$$
whenever $Z_I [h(s,\cdot)]<\infty$.
\end{enumerate}
\end{Prop}
Note that in Proposition~\ref{prop:chang_signR}, the fact that the zeros of $h$ do not accumulate insures that $Z_I [h (t,\cdot)]$ is well-defined for all $t >0$.

\bibliography{biblioBG}
\bibliographystyle{plain}

\end{document}